\documentclass[11pt,a4paper]{article}
\usepackage{amsmath}
\usepackage{amsthm}
\usepackage{amssymb}
\usepackage{amscd}
\usepackage{graphicx}
\usepackage{epsfig}
\usepackage[matrix,arrow,curve]{xy}
\usepackage{color}
\usepackage{mathrsfs}

\usepackage{bbm}
\usepackage{stmaryrd}
\usepackage{hyperref}

\usepackage{latexsym}
\usepackage{amsfonts}
\input xy

\newtheorem{theorem}{Theorem}[section]
\newtheorem{proposition}[theorem]{Proposition}
\newtheorem{lemma}[theorem]{Lemma}
\newtheorem{corollary}[theorem]{Corollary}

\theoremstyle{definition}
\newtheorem{example}[theorem]{Example}
\newtheorem{remark}[theorem]{Remark}
\newtheorem{definition}[theorem]{Definition}

\font\black=cmbx10 \font\sblack=cmbx7 \font\ssblack=cmbx5 \font\blackital=cmmib10  \skewchar\blackital='177
\font\sblackital=cmmib7 \skewchar\sblackital='177 \font\ssblackital=cmmib5 \skewchar\ssblackital='177
\font\sanss=cmss11 \font\ssanss=cmss8 scaled 900 \font\sssanss=cmss8 scaled 600 \font\blackboard=msbm10
\font\sblackboard=msbm7 \font\ssblackboard=msbm5 \font\caligr=eusm10 \font\scaligr=eusm7 \font\sscaligr=eusm5

\font\bsymb=cmsy10 scaled\magstep2
\def\all#1{\setbox0=\hbox{\lower1.5pt\hbox{\bsymb
       \char"38}}\setbox1=\hbox{$_{#1}$} \box0\lower2pt\box1\;}
\def\exi#1{\setbox0=\hbox{\lower1.5pt\hbox{\bsymb \char"39}}
       \setbox1=\hbox{$_{#1}$} \box0\lower2pt\box1\;}

\def\tx#1{{\fam0\relax#1}}

\newfam\bifam
\textfont\bifam=\blackital \scriptfont\bifam=\sblackital \scriptscriptfont\bifam=\ssblackital
\def\bi#1{{\fam\bifam\relax#1}}

\newfam\blfam
\textfont\blfam=\black \scriptfont\blfam=\sblack \scriptscriptfont\blfam=\ssblack

\newfam\bbfam
\textfont\bbfam=\blackboard \scriptfont\bbfam=\sblackboard \scriptscriptfont\bbfam=\ssblackboard

\newfam\ssfam
\textfont\ssfam=\sanss \scriptfont\ssfam=\ssanss \scriptscriptfont\ssfam=\sssanss
\def\sss#1{{\fam\ssfam\relax#1}}

\newfam\clfam
\textfont\clfam=\caligr \scriptfont\clfam=\scaligr \scriptscriptfont\clfam=\sscaligr

\def\pmb#1{\setbox0\hbox{${#1}$} \copy0 \kern-\wd0 \kern.2pt \box0}
\def\pmbb#1{\setbox0\hbox{${#1}$} \copy0 \kern-\wd0
      \kern.2pt \copy0 \kern-\wd0 \kern.2pt \box0}
\def\pmbbb#1{\setbox0\hbox{${#1}$} \copy0 \kern-\wd0
      \kern.2pt \copy0 \kern-\wd0 \kern.2pt
    \copy0 \kern-\wd0 \kern.2pt \box0}
\def\pmxb#1{\setbox0\hbox{${#1}$} \copy0 \kern-\wd0
      \kern.2pt \copy0 \kern-\wd0 \kern.2pt
      \copy0 \kern-\wd0 \kern.2pt \copy0 \kern-\wd0 \kern.2pt \box0}
\def\pmxbb#1{\setbox0\hbox{${#1}$} \copy0 \kern-\wd0 \kern.2pt
      \copy0 \kern-\wd0 \kern.2pt
      \copy0 \kern-\wd0 \kern.2pt \copy0 \kern-\wd0 \kern.2pt
      \copy0 \kern-\wd0 \kern.2pt \box0}

\mathchardef\za="710B  
\mathchardef\zb="710C  
\mathchardef\zg="710D  
\mathchardef\zd="710E  
\mathchardef\zve="710F 
\mathchardef\zz="7110  
\mathchardef\zh="7111  
\mathchardef\zvy="7112 
\mathchardef\zi="7113  
\mathchardef\zk="7114  
\mathchardef\zl="7115  
\mathchardef\zm="7116  
\mathchardef\zn="7117  
\mathchardef\zx="7118  
\mathchardef\zp="7119  
\mathchardef\zr="711A  
\mathchardef\zs="711B  
\mathchardef\zt="711C  
\mathchardef\zu="711D  
\mathchardef\zvf="711E 
\mathchardef\zq="711F  
\mathchardef\zc="7120  
\mathchardef\zw="7121  
\mathchardef\ze="7122  
\mathchardef\zy="7123  
\mathchardef\zf="7124  
\mathchardef\zvr="7125 
\mathchardef\zvs="7126 
\mathchardef\zf="7127  
\mathchardef\zG="7000  
\mathchardef\zD="7001  
\mathchardef\zY="7002  
\mathchardef\zL="7003  
\mathchardef\zX="7004  
\mathchardef\zP="7005  
\mathchardef\zS="7006  
\mathchardef\zU="7007  
\mathchardef\zF="7008  
\mathchardef\zW="700A  

\newcommand{\be}{\begin{equation}}
\newcommand{\ee}{\end{equation}}

\newcommand{\lra}{\longrightarrow}
\newcommand{\bea}{\begin{eqnarray}}
\newcommand{\eea}{\end{eqnarray}}
\newcommand{\beas}{\begin{eqnarray*}}
\newcommand{\eeas}{\end{eqnarray*}}
\def\*{{\textstyle *}}

\newcommand{\nn}{\nonumber}
\newcommand{\ot}{\otimes}

\newcommand{\pa}{\partial}
\newcommand{\ti}{\times}

\newcommand{\Ll}{{\pounds}}

\def\ran{\rangle}

\def\cO{{\cal O}}

\def\cF{\mathcal{F}}

\def\Sec{\sss{Sec}}

\def\bd{{\bi d}}

\def\sH{{\sss H}}

\def\sJ{{\sss J}}

\def\sT{{\sss T}}
\def\sV{{\sss V}}

\def\xd{\tx{d}}

\def\dt{\xd_{\sss T}}

\newdir{|>}{%
!/4.5pt/@{|}*:(1,-.2)@^{>}*:(1,+.2)@_{>}}

\newcommand{\la}{\langle}

\newcommand{\N}{\mathbb{N}}
\newcommand{\Z}{\mathbb{Z}}
\newcommand{\R}{\mathbb{R}}

\newcommand{\n}{\nabla}

\def\deg{\sss{deg}}

\newcommand{\cbl}{\color{blue}}

\newcommand{\Lie}{\textnormal{Lie}}

\newcommand{\rmd}{\textnormal{d}}

\def\cm{\color{magenta}}

\def\VB{{\rm V\!B}}

\newcommand{\wu}{\operatorname{deg}}

\DeclareMathOperator{\GL}{GL}

\DeclareMathOperator{\Ker}{Ker}

\DeclareMathOperator{\Exp}{Exp}
\newcommand{\GrL}{\operatorname{GL}}
\newcommand{\ZGrL}{\operatorname{ZGL}}
\newcommand{\we}{\wedge}
\newcommand{\wgh}{{\sss w}}

\begin{document}
\title{\bf VB-structures and generalizations\thanks{Research of JG founded by the  Polish National Science Center grant
under the contract number 2016/22/M/ST1/00542.
 }}
\date{}
\author{\\ Katarzyna  Grabowska$^1$\\ Janusz Grabowski$^2$ (Corresponding author)\\ Zohreh Ravanpak$^2$
        \\ \\
         $^1$ {\it Faculty of Physics}\\
                {\it University of Warsaw}\\
                \\$^2$ {\it Institute of Mathematics}\\
                {\it Polish Academy of Sciences}
                }
\maketitle

\begin{abstract} Motivated by properties of higher tangent lifts of geometric structures, we introduce concepts of \emph{weighted structures} for various geometric objects on a manifold $F$ equipped with a \emph{homogeneity structure}. The latter is a smooth action  on $F$ of the monoid $(\R,\cdot)$ of multiplicative reals. Vector bundles are particular cases of homogeneity structures and weighted structures on them we call $\VB$-structures. In  the case of Lie algebroids and Lie groupoids, the weighted structures include the concepts of $\VB$-algebroids and $\VB$-groupoids, intensively studied recently in the literature. Investigating various weighted structures, we prove some interesting results about their properties.

\par

\end{abstract}


\section{Introduction}
This is a conceptual paper introducing and studying some concepts related to graded differential geometry, in particular generalizations of $\VB$ structures.
The original concept of a $\VB$-algebroid was introduced by Pradines \cite{Pradines:1974,Pradines:1988} and it has been
further studied by Mackenzie \cite{Mackenzie:2005} and Gracia-Saz  \& Mehta \cite{Gracia:2010}, among others.
The concept of a $\VB$-groupoid one can find already in \cite{Mackenzie:1992,Mackenzie:2000} and
\cite[Section 2.1]{Mackenzie:2005}, where they are understood as double Lie groupoids for which one structure is a vector bundle.
$\VB$-algebroids and $\VB$-groupoids have turned out  to be especially important in the
infinitesimal description of Lie groupoids equipped with multiplicative geometric
structures and  as geometric models for representations up to homotopy \cite{Gracia-Saz:2010,Gracia-Saz:2017}.
The original definitions are quite complicated and refer to $\VB$-groupoids ($\VB$-algebroids) as  Lie groupoid (Lie algebroid) objects in the category of vector bundles.

These concepts were generalized in \cite{Bruce:2016} in much simpler terms  by using so called \emph{homogeneity structures} introduced by Grabowski and Rotkiewicz \cite{Grabowski:2009,Grabowski:2012}. Roughly speaking, a homogeneity structure on a manifold $F$ is a smooth action $h:\R\ti M\to M$  on $M$ of the monoid $(\R,\cdot)$ of multiplicative reals: $h_t\circ h_s=h_{ts}$.  Contrary to actions of the additive group $(\R,+)$ of reals, a homogeneity structure is  very  rigid. A fundamental result of \cite{Grabowski:2012} says that there are coordinate systems $(x^i)$ on $F$ such that $h_t(x^i)=(t^{w_i}x^i)$, where $w_i\ge 0$ are called the  \emph{degree} (or \emph{weight}) of the coordinate $x^i$ and $x^i$ takes  values in the whole $\R$ if only $w_i>0$. The highest $w_i$ is called the \emph{degree of the homogeneous structure}. It is clear that $F$ is a fibration over the manifold $M=h_0(F)$ (as local coordinates there can serve those $x^i$ which have degree $0$) with the typical fiber $\R^d$.  Of course, the transition maps respect the fibration structure and the degrees of coordinates. Such structures were called \emph{graded bundles} in \cite{Grabowski:2012} and the main result of \cite{Grabowski:2012} simply says that the categories of homogeneous structures and graded bundles are isomorphic.
Natural examples of graded bundles are e.g. higher tangent bundles $\sT^kM$. They can be used in a geometric interpretation of Lagrangian systems with higher order Lagrangians.

According to the description of vector bundles in \cite{Grabowski:2009}, graded bundles (homogeneous structures) of  degree one are simply vector bundles over $M$, so the concept of a graded bundle is a  natural generalization of a vector bundle. For instance, this allows a simple definition of \emph{double vector bundles} as two commuting homogeneity structures of degree one. In this language, vector bundle morphisms are just smooth maps between vector bundles  that intertwine the corresponding actions of $\R$, and vector subbundles of a vector bundle $E$ are just submanifolds which are invariant with respect to the scalar multiplication. This is much simpler than the standard concepts, as we can completely forget the addition in vector bundles. The concept of a double vector bundle can be generalized to a concept of \emph{double graded bundles} (and even \emph{$n$-tuple graded bundles}) in an obvious way. Double graded bundles in which one homogeneity structure is of  degree one (a vector bundle) we call \emph{graded-linear bundles},
$\GrL$-bundles in short. They are in a sense $\VB$-graded bundles. An important fact is that homogeneity structures on $F$ can be lifted to $\sT F$ and $\sT^*F$ making them into $\GrL$-bundles.

We extend the concept of a graded bundle to the concept of a \emph{$\Z$-graded bundle} allowing in the definition of the graded bundle for weights of negative degrees.  A $\Z$-graded bundle induces an analog of a homogeneity structure, the so called \emph{$\Z$-homogeneity structure}. We can further define \emph{double $\Z$-graded bundles} and \emph{$\Z$-graded-linear bundles} ($\ZGrL$-bundles). Moreover, a $\Z$-graded bundle structure on $M$ induces canonical $\ZGrL$-structures on $\sT M$ and $\sT^*M$.

We  also prove that a $\ZGrL$-structure on a manifold $F$ induces canonically a $\ZGrL$-bundle structure on $F^*$ (of course, the duality is with respect to the vector bundle structure) and that this is a true duality, $(F^*)^*=F$. We define the tensor product of $\Z$-graded bundles and the degree of their sections. In the case of homogeneous tensor fields on a $\Z$-graded bundle, the degree of tensors coincides with their degrees as sections of the corresponding tensor bundles.

 In \cite{Bursztyn:2016} it was discovered that the use of vector bundle characterization in terms of homogeneity structures of  degree one \cite{Grabowski:2009} substantially simplifies the definition of $\VB$-algebroids and $\VB$-groupoids. The new definition says that a $\VB$-algebroid ($\VB$-groupoid) is a Lie algebroid (Lie groupoid) $F$ equipped additionally with a vector bundle structure (i.a. a homogeneity structure of degree $1$) such that the maps $h_t:F\to F$ are Lie algebroid (Lie groupoid) morphisms for all $t\in\R$.

In \cite{Bruce:2016} we introduced an obvious generalization of the  above concepts of $\VB$-algebroids and $\VB$-groupoids,
by skipping the assumption that the homogeneity structure $h$ is of  degree one. The generalized objects were called \emph{weighted algebroids} and \emph{weighted groupoids}. Natural examples are higher tangent bundles $\sT^kE$ and $\sT^kG$ of  Lie algebroid $E$ and Lie groupoid $G$, respectively. The word `weighted' was chosen because  graded Lie algebroids have already a different meaning in the literature.

 In this paper, we introduce and study further concepts of weighted structures on a graded bundle $F$, such as
\begin{itemize}
\item weighted tensor fields and distributions;
\item weighted Nijenhuis structures, weighted (almost) complex structures, weighted product and tangent manifolds;
\item weighted foliations and fibrations;
\item weighted Ehresmann connections;
\item weighted Poisson, symplectic and pseudo-Riemannian structures;
\item weighted contact structures;
\item weighted Poisson-Nijenhuis structures;
\item weighted principal bundles.
\end{itemize}
The \emph{weighted structures} are understood as geometric structures compatible with the homogeneity structure on $F$;  what compatibility means is precisely explained in each case. If a given geometric structure is compatible with a vector bundle structure (homogeneity structure of degree one), then we speak about \emph{$\VB$-structures}.
For most of the weighted  structures, we make `intelligent guesses'  what compatibility means. It depends on considering  canonical lifts of the structures to the higher tangent bundles $\sT^rM$ (which are canonically graded bundles) as `compatible' with the graded bundle structure. In particular, we compute the degrees of  the lifted tensors and we show that the higher lifts of vector-valued differential forms respect the \emph{Fr\"olicher-Nijenhuis} and \emph{Nijenhuis-Richardson} brackets. This immediately implies that the higher tangent lifts of Nijenhuis tensors are Nijenhuis tensors, higher tangent lifts of complex structures are complex structures, etc.

The paper is organized as follows. First, we introduce the concept of $\Z$-graded bundles and the corresponding $\Z$-homogeneity structures, generalizing the concepts of graded bundles (i.e. $\N$-graded bundles) and homogeneity structures as  they appeared in \cite{Grabowski:2009,Grabowski:2012}. We show that the concept of homogeneity is much weaker in the $\Z$-graded case, allowing for functions of arbitrary real degree. Then, we present the concepts of double graded bundles, graded-linear ($\GrL$) bundles, i.e. $\VB$-graded bundles, tensor products of graded bundles and tangent and phase lifts of homogeneity structures (see \cite{Bruce:2015,Bruce:2015a,Bruce:2016,Bruce:2017a,Grabowski:2009, Grabowski:2012}), etc., extending all these notions to $\Z$-graded case.

Further, we study the duality for $\ZGrL$-bundles and we describe the degree of their sections. In Section 4, we discuss higher tangent lifts of geometric structures as they are presented in \cite{Wamba:2011} and \cite{Morimoto:1970c},  they are used as motivating examples to define weighted structures. In particular, we study higher tangent lifts of vector-valued differential forms with respect the  Fr\"olicher-Nijenhuis and Nijenhuis-Richardson brackets.

Finally, in Section 5, we introduce and study various weighted structures  and discuss also some natural examples.

\section{Graded bundles and homogeneity structures}
\subsection{$\Z$-graded bundles}
According to  textbooks, a \textit{vector bundle} is a locally trivial fibration $\zt:E\to M$ which, locally over some open subsets $U\subset M$, reads $\zt^{-1}(U)\simeq U\ti\R^n$ and admits an atlas in which local trivializations transform linearly in fibers:
\begin{equation}\label{bu:1} U\cap V\ti\R^n\ni(x,y)\longmapsto(\zf(x),A(x)y)\in U\cap V\ti\R^n\,,\quad
A(x)\in\GL(n,\R).\end{equation}
This can be expressed also in terms of a gradation in which base coordinates (pull-backs of coordinates in $M$) $x=(x^i)$ have degree $0$, and linear coordinates $(y)$ have  degree one. Such coordinates on a vector bundle we will call \emph{affine}. Linearity in $y's$ of the transformation rules is now equivalent to the fact that changes of coordinates respect the degrees.  A morphism in the category of vector bundles is represented by the following commutative diagram of smooth maps
$$ \xymatrix{
E_1\ar[rr]^{\Phi} \ar[d]^{\zt_1} && E_2\ar[d]^{{\zt_2}} \\
M_1\ar[rr]^{{\varphi}} && M_2 }
$$
being linear (homogeneous) in fibres, i.e. preserving the degrees.

A straightforward generalization of the above concept is the following (cf. \cite{Bruce:2017a,Grabowski:2012}).
Consider a graded vector space $\R^\bd=\R^{d_1}\oplus\cdots\oplus\R^{d_k}$, where $\bd=(d_1,\dots,d_k)$, with positive integers $d_i$, and equipped with a vector field $\n$ of the form
\be\label{nb}
\n=\sum_{i=1}^kw_i\,\sum_{a=1}^{d_i}y_i^a\pa_{y_i^a}\,.
\ee
Here, $y_i=(y_i^1,\dots,y_i^{d_i})$ are canonical coordinates in $\R^{d_i}$ and $w_i$ are non-zero integers, $i=1,\dots,k$. Fixing such a vector field is equivalent to fixing $\wgh=(w_1,\dots,w_k)$
understood as the vector of \emph{degrees} (\emph{weights}) of the subspaces $\R^{d_1},\dots,\R^{d_k}$, making $\R^\bd$ a $\Z$-graded vector space which we denote $\R^\bd(\wgh)$. The vector field $\n$ will be called the \emph{weight vector field}. It induces
the notion of a \emph{homogeneity} for smooth functions on $\R^\bd$.
\begin{definition}
A smooth function $f$ on $\R^\bd$ is called \emph{homogeneous of degree (weight) $w\in\R$} if
\be\label{hm}\n(f)=w\,f\,.\ee
\end{definition}
\noindent By the \emph{degree} of the weight vector field we will understand $\wu(\n)=\max{|w_i|}$. One can easily check that the coordinate $y_i^a$ is homogeneous of weight $w_i$. It is also easy to see that (\ref{hm}) is equivalent to
$$ f\circ h_t=t^w\cdot f\,,\ t>0\,,$$
where
\be\label{hm1}
h_t(y)=\left(t^{w_1}\,y_1,\dots,t^{w_k}\,y_k\right)\,,\ t>0\,,
\ee
is the smooth action of the multiplicative group $\R^\ti=\R\setminus\{ 0\}$ of non-zero reals,
induced by the weight vector field $\n$.

\medskip
Let us fix now $\bd=(d_1,\dots,d_k)$, $\wgh=(w_1,\dots,w_k)$, and consider a fiber bundle $\zt:F\to M$ with the typical fiber $\R^\bd(\wgh)$ and a local trivializations
$$\zf_\za:\zt^{-1}(U_\za)\to U_\za\ti\R^\bd(\wgh)\,,$$
where $\{ U_\za\}$ is an open covering of $M$ with coordinate charts $(U_\za,x_\za)$. On each $U_\za\ti\R^\bd(\wgh)$ with coordinates $(x_\za,y_\za)$ we have a canonical vector field $\n_\za$ which formally reads as (\ref{nb}):
$$
\n_\za=\sum_{i=1}^kw_i\,\sum_{a=1}^{d_i}(y_\za)_i^a\pa_{(y_\za)_i^a}
$$
(it is therefore vertical). It defines the notion of homogeneity of a smooth function for which coordinates $x_\za^A$ on $M$ are of degree 0 and coordinates $(y_\za)_i^a$ have weights as $(y_i^a)$ in $\R^\bd(\wgh)$.

Finally, let us assume that the fiber bundle transition maps,
\bea\nn&\zF_{\za\zb}=\zf_\za\circ\zf^{-1}_\zb:(U_\za\cap U_\zb)\ti\R^\bd(\wgh)\lra (U_\za\cap U_\zb)\ti\R^\bd(\wgh)\,,\\
\label{tm}&\zF_{\za\zb}(x_\zb,y_\zb)=(\zf_{\za\zb}^1(x_\za),\zf_{\za\zb}^2(x_\zb,y_\zb))\,,
\eea
preserve the weights of coordinates, thus the weights of all homogeneous functions, i.e. transfer $\n_\zb$ into $\n_\za$. This is equivalent to the fact that
$\zF_{\za\zb}$ respect the corresponding actions of the multiplicative group of positive reals,
$$\zF_{\za\zb}\circ h_t^\zb=h_t^\za\circ\zF_{\za\zb}\,,\ t>0\,,$$
where
\be\label{e9}h_t^\za(x_\za,y_\za)=(x_\za,t^{w_i}\,(y_\za)_i)\,.\ee
Note that $h_t\circ\zt=\zt\circ h_t=\zt$. Respecting local weight vector fields by the transition maps implies that the family $\{\n_\za\}$ gives rise to a globally defined weight vector field $\n_F$ (or globally defined action $h^F_t$ of the multiplicative group of positive reals). Note that $M$ is canonically a submanifold of $F$.
Indeed, locally we can view $U_\za$ as embedded in $F$ as $\zf_\za^{-1}(U_\za\ti\{ 0\})$. But transition maps respect the local weight vector fields whose zeros form are $U_\za\ti\{ 0\}$, so that these embedding of $U_\za$ into $F$ give rise to an embedding of $M$. This is not a standard property of fiber bundles.

Any local trivialization of a fiber bundle $\zt:F\to M$, with the typical fiber $\R^n$, of the form
$U\ti\R^n$, where $U$ is an open subset of $\R^N$ and $U\ti\R^n$ (with canonical coordinates $(x^i, y^a)$) which is equipped with a weight vector field
\be\label{lwvf}\n=\sum_aw_a\,y^a\pa_{y^a}\,,\ w_a\in\Z^\ti=\Z\setminus\{ 0\}\,,\ee
we will call a \emph{$\Z$-chart}. The above construction shows how a proper gluing  of $\Z$-charts, i.e. a gluing respecting the local weight vector fields, leads to a global geometric object.
\begin{definition}
A fiber bundle $\zt:F\to M$ with the typical fiber $\R^n$ and an atlas of local trivializations with $\Z$-charts, whose gluing by transition maps respects the local weight vector fields (\ref{lwvf}), is called a \emph{$\Z$-graded bundle}. The \emph{degree} of a $\Z$-graded bundle is the degree of local weight vector fields (all are the same).
\end{definition}
\begin{remark} As we have an atlas for $F$ consisting of $\Z$-charts,  we will work only with local coordinates which have only integer weights. This is important, as on a $\Z$-chart smooth functions with arbitrary real weight could exist (see Example \ref{ex0}).
 Note also that in the case when all weights $w_a$ are positive, we recover the concept of a \emph{graded bundle} (we can call it here \emph{$\N$-graded bundles}) introduced in \cite{Bruce:2017a}.
\end{remark}
\noindent From our previous considerations we get the following.
\begin{proposition}
Every $\Z$-graded bundle $\zt:F\to M$ is canonically equipped with a globally defined \emph{weight vector field} $\n_F$ which locally, in $\Z$-charts, looks like (\ref{lwvf}).
The weight vector field induces also a smooth action $h^F_t$, $t\ne 0$, of the multiplicative group $\R^\ti$ of multiplicative reals, which in $\Z$-chart coordinates $(x^i,y^a)$ reads
\be\label{haction}
h^F_t(x^i,y^a)=(x^i,t^{w_a}\,y^a)\,.
\ee
If we use the convention that $0^w=0$ for $w\ne 0$, the above formula  defines actually an action
$h$ of the multiplicative monoid $(\R,\cdot)$ of reals:
\be\label{ma} h^F:\R\ti F\to F\,,\ h^F(t,p)=h^F_t(p)\,,\ h^F_t\circ h^F_s=h^F_s\circ h^F_t\,.\ee
This monoid action is smooth if and only if all weights $w_a$ are positive.
Moreover, the manifold $M$ can be viewed as a submanifold in $F$ by a canonical identification of $M$ with $h_0(F)$.
\end{proposition}
\begin{definition} We call a (local) function $f$ on $F$ \emph{homogeneous of weight $w\in\R$} if $\n_F(f)=wf$
or, equivalently,
$$f\circ h_t=t^wf\,,\ t>0\,.$$
\end{definition}
\begin{definition}
By \emph{$\Z$-homogeneity structure} we understand an action $h:\R\ti F\to F\,,\ h(t,p)=h_t(p)$ of the multiplicative monoid $(\R,\cdot)$ of reals on a fiber bundle $\zt:F\to M$ with the typical fiber $\R^n$ for which there is a covering of $F$ by local trivialization charts $\zt^{-1}(U)\simeq U\ti\R^n$ in which $h_t$ read as in (\ref{haction}).
\end{definition}
It immediately follows that $M\simeq h_0(F)$. Of course, the action $h^F$ of the multiplicative monoid $(\R,\cdot)$ we defined above for a $\Z$-graded bundle is a $\Z$-homogeneity structure which we call \emph{associated with the $\Z$-graded bundle $\zt:F\to M$}.

\medskip\noindent There are clear notions of morphisms of $\Z$-graded bundles and $\Z$-homogeneity structures.
\begin{definition} A \emph{morphism of $\Z$-graded bundles} $\zt_i:F_i\to M_i$, $i=1,2$, is a morphism $\zf:F_1\to F_2$ of the corresponding fiber bundles preserving homogeneity, i.e. such that the pull-backs of homogeneous functions on $F_2$ of weight $w$ are homogeneous functions on $F_1$ of weight $w$ (or equivalently, the vector fields $\n_{F_1}$ and $\n_{F_2}$ are $\zf$-related).

A \emph{morphism of $\Z$-homogeneity structures} $h^i$ on fiber bundles $F_i$, $i=1,2$ are smooth maps $\zf:F_1\to F_2$ intertwining $h^{i}$, $i=1,2$, i.e.
\be\label{mor}
\zf\circ h^{F_1}_t=h^{F_2}_t\circ\zf
\ee
for all $t\in\R$.
\end{definition}
It is easy to see that we obtain in this way the categories of $\Z$-graded bundles and $\Z$-homogeneity structures.
In fact, the following is nearly obvious.
\begin{proposition}\label{Zequi}
The categories of $\Z$-graded bundles and $\Z$-homogeneity structures are canonically equivalent.
\end{proposition}
\begin{example}
Consider $\R^{(1,1)}$ with coordinates $(y,z)$, where $y$ is of degree 1 and $z$ is of degree 2.
The map
$$\phi:\R^{(1,1)}\to\R^{(1,1)}\,,\quad (y,z)\mapsto (y,z+y^2)$$
is a morphism of $\N$-graded bundles (over a point in this case) but is not linear, i.e. it does not preserve the structure of the graded vector space
$$\R^{(1,1)}=\R\oplus\R=\langle y\rangle\oplus\langle z\rangle\,.$$
This shows the difference between the categories of $\Z$-graded bundles and $\Z$-graded vector bundles.
\end{example}
\begin{remark}
To simplify the notation, on a $\Z$-graded bundle $F\to M$ of degree $k$ we will usually use a systems of homogeneous local coordinates $(x^i)$ assuming by default that the weight of variable $x^i$ is $w_i\in\Z$,\, $-k\le w_i\le k$. Such a system of local coordinates on $F$ induces canonically a system $(x^i)_{w_i=0}$ of local coordinates on $M$. This is a convenient notations, since the weight vector field $\n_F$ in such coordinates reads
\be\label{wvf}\nabla_F=\sum_iw_i\,x^i\pa_{x^i}\,.\ee
Note that the $\Z$-graded bundles are purely even graded manifolds in the sense of Voronov \cite{Voronov:2002}.
\end{remark}
\begin{example}
The tangent bundle $F=\sT M$ of a manifold $M$ is a vector bundle which is a $\Z$-graded bundle with homogeneous adapted (from $M$) standard local coordinates $(x^i,\dot x^j)$, where $x^i$ are of degree $0$ and $\dot x^j$ are of degree $1$.
The cotangent bundle $\sT^*M$ is again a vector bundle with the dual coordinates $(x^i,p_j)$, but is convenient to take the degree $0$ for $x^i$ and degree $-1$ for $p_j$, that turns $\sT^*M$ into a $\Z$-graded bundle. In this case the pairing between $\sT M$ and $\sT^*M$ is of degree 0, which supports the standard convention $\la\pa_{x^i},\xd x^i\ran=1$, independents of the degree of $x^i$.
\end{example}

\begin{remark}
There is much deeper result \cite{Grabowski:2012} for $\N$-graded bundles than Proposition \ref{Zequi}. The homogeneity structures associated with $\N$-graded bundles are smooth actions of the monoid $(\R,\cdot)$ on $F$. The main result in
\cite{Grabowski:2012} states that any such a smooth action on a manifold $F$ is associated with a $\N$-graded bundle on $h_0:F\to M=h_0(F)$. We do not need any additional assumptions, e.g. that the manifold $F$ is a fiber bundle, etc.
We discuss these questions in the next subsection.
\end{remark}
\begin{example}\label{ex0}
It is interesting that if the degrees of coordinates have  different signs, then there exist local smooth functions on $F$ which are homogeneous of arbitrary degree $a\in\R$ and which, even for an integer degree $a$, are not polynomials in variables $x^i$, $w_i\ne 0$.

Take for example $\R^2$ with coordinates $(x,y)$, where $x$ is of degree $1$ and $y$ is of degree $-1$, $h_t(x,y)=(tx,t^{-1}y)$. Take a nonzero function $\zf:\R\to\R$, $\zf(0)=0$, which is flat at $0$ (all derivatives at $0$ vanish) but not constant. Then, $f(x,y)=\zf(xy)$ is of degree $0$ with respect to $h_t$ but is not constant. The function $f_1(x,y)=x\cdot\zf(xy)$ is of degree $1$ but is not a polynomial in coordinates while the function
$$\begin{cases}f_d(x,y)=|x|^d\zf(xy)\quad\text{for}\quad x\ne 0\\
f_d(x,y)=0 \qquad\text{for}\quad x=0
\end{cases}$$
is clearly smooth ($\zf$ is flat at $0$) and of degree $d\in\R$.
\end{example}
 As shown in the next theorem, such strange homogeneous functions must be flat at $0$ on fibers of $F$.
\begin{theorem}
If a  smooth function $f:\R^n\to\R$ on a $\Z$-graded bundle $\R^n$ with non-zero weights is homogeneous of degree $w\in\R$ and not flat at $0$, then $w\in\Z$.
\end{theorem}
\begin{proof}
Take $\R^n$ with canonical coordinates $(x^i)$ such that $h_t(x^i)=(t^{w_i}x^i)$, where $w_i\in\Z^\ti$. Suppose $f$ is a smooth function on $\R^n$ such that $f\circ h_t=t^wf$ for $t>0$. As $f$ is not flat at $0$ there is a Taylor decomposition $f(x)=P(x)+o(x)$ of $f$ around $0$ with $P$ being a non-zero polynomial of degree $\le r$ and $\lim_{x\to 0}(o(x)/|x|^r)=0$. We have $P\circ h_t +o\circ h_t=t^w(P+o)$ for $t>0$. Since for fixed $t>0$, the functions $o\circ h_t$ and $t^wo$ are also appropriately small near $0$ and the Taylor polynomial of a given rank is uniquely determined,
we have $P\circ h_t=t^wP$ for $t>0$. Because any polynomial in variables $x^i$ is of integer homogeneous degree, we have $w\in\Z$.

\end{proof}

We can easily extend the concept of homogeneity from functions to arbitrary tensor fields on the $\Z$-graded bundle $F$: a tensor $K$ is homogeneous of degree $w$ with respect to $h$ if $\Ll_{\nabla_F}(K)=w\cdot K$.
For instance, a vector field $Y$ is of degree $w$ if and only if
\be\label{e4}\Exp(t\nabla_F)_*(Y)=(h_{\exp(t)})_*(Y)=e^{-tw}\,Y\,,\ee
where $\Exp(t\nabla_F)$ is the flow induced by $\n_F$ (any weight vector field is complete).
This is because in general
$$\left.\frac{\rmd}{\rmd t}\right|_{t=0}\Exp(tX)_*(Y)=[Y,X]\,.$$
Note that (\ref{e4}) is equivalent to
$$ (h_t)_*(Y)=t^{-w}\cdot Y$$
for $t> 0$.

\noindent
Similarly, a differential form $\zw$ is of degree $w$ if and only if
$$(h_t)^*(\zw)=t^{w}\zw$$
for $t>0$.

\begin{example} If $(x^i)$ are homogeneous coordinates, then the vector field $\pa_{x^j}$ is of degree $-w_j$.
Indeed,
$$[\nabla_F,\pa_{x^j}]=[\sum_iw_i\,x^i\pa_{x^i},\pa_{x^j}]=-w_j\pa_{x^j}\,.$$
Similarly, the one-form $\xd x^j$ is of degree $w_j$:
$$\Ll_{\nabla_F}(\xd x^j)=\xd(i_{\nabla_F}\xd x^j)=\xd(w_jx^j)=w_j\xd x^j\,.$$
\end{example}

\subsection{Graded bundles}
Graded bundles form  a particular and very important class of $\Z$-graded bundles with many nice properties, which will be the main geometric structure of the paper.
\begin{definition} If all weights $w_i$ of coordinates in a $\Z$-graded bundle $F$ are non-negative, we speak just about a \emph{$\N$-graded bundle} or simply a \emph{graded bundle} (see \cite{Bruce:2017a,Grabowski:2012}).
\end{definition}
\noindent In this case the corresponding action $h=h^F:\R\ti F\to F$ of the monoid $(\R,\cdot)$ is smooth. Of course, graded bundles of degree 1 are exactly vector bundles.
\begin{theorem}[Grabowski-Rotkiewicz \cite{Grabowski:2012}] Homogeneous functions on graded bundles $\zt:F\to M$ are locally polynomials in homogeneous coordinates of non-zero degree with basic functions as coefficients . In consequence, the transformations of fiber coordinates $A(x,y)$ in (\ref{bu:1}) must be polynomial  in the homogeneous fiber coordinates $y_j$'s, i.e. any graded bundle is a \textit{polynomial bundle}.
\end{theorem}
\noindent Note that the above theorem is not valid in the case of general $\Z$-graded bundles.
Using now homogeneous fiber bundle coordinates $(x^i)$ on $F$ (they do not denote coordinates on $M$ any longer), we have $x^i\circ h_t=t^{w_i}x^i$ also for $t<0$. Moreover, homogeneous functions $f$ on $F$ may have only non-negative integer degrees $w$ \cite{Grabowski:2012} and $f\circ h_t=t^wf$ also for $t<0$.
The weight vector field $\nabla _F$ has formally the same form (\ref{wvf}), but all  $w_i$ are non-negative.

We define homogeneous tensors on graded bundles as in the case of $\Z$-graded bundles.
For instance, a vector field $Y$ is of degree $w$ if and only if
$$ (h_t)_*(Y)=t^{-w}\cdot Y$$
for all $t\ne 0$ and a differential form $\zw$ is of degree $w$ if and only if
$$(h_t)^*(\zw)=t^{w}\zw$$
for all $t\ne 0$.

\begin{example}\label{e3}(\cite{Grabowski:2012})
Consider the second-order tangent bundle $\sT^2M=\sJ^2_0(\R,M)$, i.e. the bundle of second jets of smooth maps $(\R,0)\to M$. Writing Taylor expansions of curves in local coordinates $(x^A)$ on $M$:
\begin{equation*}x^A(t)=x^A(0)+\dot x^A(0)t+\ddot x^A(0)\frac{t^2}{2}+o(t^2)\,,\end{equation*}
we get local coordinates $(x^A,\dot x^B,\ddot x^C)$ on $\sT^2M$, which transform as
\beas
x'^A&=&x'^A(x)\,,\\
\dot x'^A&=&\frac{\pa x'^A}{\pa x^B}(x)\,\dot x^B\,,\\
\ddot x'^A&=&\frac{\pa x'^A}{\pa x^B}(x)\,\ddot x^B+\frac{\pa^2 x'^A}{\pa x^B\pa x^C}(x)\,\dot x^B\dot x^C\,.
\eeas
This shows that associating with $(x^A,\dot x^B,\ddot x^C)$ the weights $0,1,2$, respectively, will give us a graded bundle structure of degree $2$ on $\sT^2M$. Note that, due to the quadratic terms above, this is not a vector bundle over $M$.
All this can be  generalized to higher tangent bundles $\sT^kM=\sJ^k_0(\R,M)$. The adapted coordinate systems are
$(x^A,x^B_i)$, $i=1,\dots,k$, where $x^A$ are of degree 0 and $x^B_i$ are of degree $i=1,\dots,k$.
\end{example}
\begin{remark}
Note that there is an alternative convention for canonical coordinates. It is used e.g.  in the paper \cite{Morimoto:1970c} by Morimoto which  will be our main reference in the next section.
We write a curve in coordinates on $M$ as
$$x^A(t)=x^A_{0}(0)+t\,x^A_{1}(0)+t^2\,x^A_{2}(0)+\cdots+t^n\,x^A_{n}(0)+  o(t^r)\,.$$
This leads to local coordinates $(x^A_{0},x^B_{1},\dots,x^Z_{r})$ on $\sT^rM$. The coordinate $x^K_{i}$ carries the weight $i$ and the transition functions look like
\beas
x'^A_{0}&=&x'^A_{0}(x_{0})\,,\\
x'^A_{1}&=&\frac{\pa x'^A_{0}}{\pa x^B_{0}}(x_{0})\, x^B_{1}\,,\\
x'^A_{2}&=&\frac{\pa x'^A_{0}}{\pa x^B_{0}}(x_{0})\, x^B_{2}+\frac{1}{2}\frac{\pa^2 x'^A_{0}}{\pa x^B_{0}\pa x^C_{0}}(x_{0})\, x^B_{1} x^C_{1}\,, etc.
\eeas
 In the following, we shall use  Morimoto's convention, since it leads to fewer numerical factors in formulae.
\end{remark}
\begin{example}(\cite{Grabowski:2014})\label{Gr}
If $\zt:E\to M$ is a vector bundle, then $\we^r\sT E$ is canonically a graded bundle of degree $r$ with respect to the projection
\begin{equation*}\we^r\sT\zt:\we^r\sT E\to \we^r\sT M\,.\end{equation*}
For $r=2$, the adapted coordinates {on $\we^2\sT E$} are $(x^\zr, y^a,{\dot x}^{\zm\zn}, y^{\zs b}, z^{cd} )$, ${\dot x}^{\zm\zn}=-{\dot x}^{\zn\zm}$, $z^{cd}=-z^{dc}$, coming from the decomposition of a bivector
\begin{equation*}\wedge ^2\sT E\ni u = \frac{1}{2} {\dot x}^{\zm\zn} \frac{\partial}{\partial x^\zm}\wedge \frac{\partial }{\partial x^\zn} + y^{\zs b}\frac{\partial }{\partial x^\zs}\wedge \frac{\partial }{\partial y^b} +\frac{1}{2} {z}^{cd} \frac{\partial }{\partial y^c}\wedge \frac{\partial }{\partial y^d}\,,
\end{equation*}
are of degrees $0,1,0,1,2$, respectively.
\end{example}

\bigskip
One can pick an atlas of $F$ consisting of charts for which the degrees of homogeneous local coordinates $(x^{A}, y_{w}^{a})$ are $\wu(x^{A}) =0$ and  $\wu(y_{w}^{a}) = w$, \ $1\leq w \leq k$, where $k$ is the degree of the graded bundle. The local changes of coordinates  are of the form
\begin{eqnarray}\nn
x'^{A} &=& x'^{A}(x),\\
\nonumber y'^{a}_{w} &=& y^{b}_{w} T_{b}^{\:\: a}(x) + \sum_{\stackrel{1<n  }{w_{1} + \cdots + w_{n} = w}} \frac{1}{n!}y^{b_{1}}_{w_{1}} \cdots y^{b_{n}}_{w_{n}}T_{b_{n} \cdots b_{1}}^{\:\:\: \:\:\:\:\:a}(x),
\end{eqnarray}
where $T_{b}^{\:\: a}$ are invertible and $T_{b_{n} \cdots b_{1}}^{\:\:\: \:\:\:\:\:a}$ are symmetric in the indices $b_1,\dots,b_n$.

In particular, the transition functions of coordinates of degree $r$ involve only coordinates of degree $\le r$, defining a reduced graded bundle $F_r$ of degree $r$ (we simply `forget' coordinates of degrees $>r$).

Transformations for the canonical projection $F_r\to F_{r-1}$ are linear modulo a shift by a polynomial in variables of degrees $<r$,
$$y'^{a}_{r} = y^{b}_{r} T_{b}^{\:\: a}(x) + \sum_{\stackrel{1<n  }{w_{1} + \cdots + w_{n} = r}} \frac{1}{n!}y^{b_{1}}_{w_{1}} \cdots y^{b_{n}}_{w_{n}}T_{b_{n} \cdots b_{1}}^{\:\:\: \:\:\:\:\:a}(x)\,,$$
so the fibrations $F_r\to F_{r-1}$ are {affine}.
 The linear part of $F_r$ corresponds to a vector subbundle
$\bar F_r$ over $M$ (we put $y^a_w$ in $F_r$, with $0<w<r$, equal to $0$).

In this way we get for any graded bundle $F$ of degree $k$, like for jet bundles, a tower of affine fibrations
\be\label{afffib}
F=F_{k} \stackrel{\tau^{k}}{\longrightarrow} F_{k-1} \stackrel{\tau^{k-1}}{\longrightarrow}   \cdots \stackrel{\tau^{3}}{\longrightarrow} F_{2} \stackrel{\tau^{2}}{\longrightarrow}F_{1} \stackrel{\tau^{1}}{\longrightarrow} F_{0} = M\,.
\ee
\begin{example}
In the case of the canonical graded bundle $F=\sT^kM$, we get exactly the tower of projections of jet bundles
$$\sT^kM \stackrel{\tau^{k}}{\longrightarrow} T^{k-1}M \stackrel{\tau^{k-1}}{\longrightarrow}   \cdots \stackrel{\tau^{3}}{\longrightarrow} \sT^{2}M \stackrel{\tau^{2}}{\longrightarrow}\sT M \stackrel{\tau^{1}}{\longrightarrow} F_{0} = M\,.
$$
\end{example}

\begin{remark} A graded bundle has an analog in supergeometry, namely $N$-manifold in the terminology of Roytenberg \cite{Roytenberg:2002} (see also \cite{Severa:2005}), where variables of odd  (even) degree have odd parity (resp., even parity). As commutation rules for these variables use the parity, the odd variables are nilpotent, and the variables of even degrees are by definition formal, this makes the theory quite different.
\end{remark}

\subsection{Homogeneity structures}
As we work with a $\N$-graded bundle, the $(\R,\cdot)$-action on $F$ is smooth, so we will borrow a definition of a \emph{$\N$-homogeneity structure} from \cite{Grabowski:2012}.
\begin{definition} A smooth action of the monoid $(\R,\cdot)$ on a manifold $F$ we will call a \emph{$\N$-homogeneity structure}. We will call usually simply a \emph{homogeneity structure}.
\end{definition}
As in general, the images of smooth projections on manifolds are smooth submanifolds \cite[Theorem 1.13]{Kolar:1996}, on a homogeneous manifold  $F$ we have a natural smooth projection $h_0:F\to M:=h_0(F)$ onto its smooth submanifold $M$.
Any graded bundle structure on $F$ uniquely induces a homogeneity structure $h^F$ which in homogeneous coordinates $(x^i)$ takes the form
$$h^F_t(x^i)=(t^{w_i}x^i)\,,$$ where $w_i\ge 0$ is the weight of $x^i$.
\begin{example}
The natural homogeneity structure $h$ on $\sT^kM=\sJ^k_0(\R,M)$ (see Example \ref{e3}) is given by $h_s([\phi]_k)=[\phi_s]_k$, where $[\phi]_k$  is the $k$-th jet of the curve $\phi:\R\to M$ at $0$ and $\phi_s(t)=\phi(st)$ (see \cite{Grabowski:2012}).
\end{example}
 \begin{proposition}[Grabowski-Rotkiewicz \cite{Grabowski:2012}]
For a homogeneity structure, only non-negative integer degrees of homogeneity are allowed. Moreover, the homogeneity structure is completely determined by $h_t$ for $t>0$. If $f$ is of weight $w$, then $f\circ h_t=t^wf$ also for $t\le 0$.
\end{proposition}

\begin{definition} Let $(F^i,h^i)$ be graded bundles for $i=1,2$. We say that a smooth map  $\Phi:F^1\to F^2$ \emph{is of degree $\zl$} if the pull-backs $f\circ\zF$ of (local) homogeneous functions $f$ of degree $w$ on $F^2$ are homogeneous of degree $w+\zl$. We call $\zF$ a \emph{morphism of graded bundles} if $\zF$ is of degree $0$.
\end{definition}
\noindent It is easy to see also the following.
\begin{proposition}
The map $\Phi:F^1\to F^2$ is a morphism of graded bundles if and only if
$h^{F^2}_t\circ\zF=\zF\circ h^{F^1}_t$, and if and only if the weight vector fields $\nabla_{F^1}$ and $\nabla_{F^2}$ are $\zF$-related.
\end{proposition}

The fundamental fact in graded bundle theory is that graded bundles and homogeneity structures are equivalent concepts. This is a non-trivial result, contrary Proposition \ref{Zequi} for $\Z$-graded bundles.
 \begin{theorem}[Grabowski-Rotkiewicz \cite{Grabowski:2012}]
Associating the homogeneity structure with a graded bundle is an equivalence of categories. In particular, for any homogeneity structure $h$ on a manifold $F$, there is a smooth submanifold $M=h_0(F)\subset F$ and a non-negative integer $k\in\mathbb N$ such that $h_0:F\to M$ is canonically a graded bundle of degree $k$ whose homogeneity structure coincides with $h$. In other words, $h_0:F\to M$ is a fibration with the typical fiber $\R^n$ and there is an atlas on $F$ consisting of local homogeneous functions $(x^i,y^j)$ on
$$(h_0)^{-1}(U)\simeq U\ti\R^n$$ such that
$$h_t(x^i,y^j)=(x^i,t^{w_j}y^j)\,,$$ where $w_j>0$ is the weight of $y^j$.
\end{theorem}
\noindent By definition, the \emph{degree of $h$} is the degree of the graded bundle $h_0:F\to M$, i.e. the biggest $w_i$. We will refer to coordinates $(x^i,y^j)$ as simply to homogeneous coordinates. The corresponding weight vector field reads
$$\n_F=\sum_jw_jy^j\,\pa_{y^j}\,.$$
However, it is sometimes convenient not to distinguish coordinates $(x^i)$ on $M$ and $(y^j)$ in the fibers.
In such cases, the coordinates $x^i$ be homogeneous coordinates on $F$ (not on $M$) with weights $w_i\ge 0$, and the coordinates on $M$ are distinguished as $(x^i)_{w_i=0}$, i.e. those $x^i$ which have weight 0. The weight vector field in such coordinates reads $\n_F=\sum_iw_ix^i\,\pa_{x^i}$ which is the same as
$$\n_F=\sum_{w_i\ne 0}w_ix^i\,\pa_{x^i}\,.$$
In the rest of the paper, we will mostly understand graded bundles as homogeneity structures.

The proposition below is obvious.
\begin{proposition}
Let $F_1\to M_1$ and $F_2\to M_2$ be graded bundles of degrees $k_1$ and $k_2$, respectively. Denote local homogeneous coordinates in $F_1$ with $(x^i)$ of weights $w_i$, and in $F_2$ with $(y^j)$ of weights $v_j$. Then, the Cartesian product $F_1\ti F_2\to M_1\ti M_2$ is canonically a graded bundle of degree $\max(k_1,k_2)$ with respect to the weight vector field $\nabla_{F_1\ti F_2}$ such that
\be\label{cprod}\nabla_{F_1\ti F_2}=(\nabla_{F_1},\nabla_{F_2})=\sum_iw_ix^i\pa_{x^i}+\sum_jv_jy^j\pa_{y^j}\,.\ee
Moreover, if $M_1=M_2=M$, then $F_1\ti_MF_2\to M$ is also canonically a graded bundle of degree $\max(w_i,)$
whose weight vector field in coordinates $\left((x^i)\,,(y^j)_{v_j>0}\right)$ reads as (\ref{cprod}). In all these cases the homogeneity structure $h^{F_1\ti_MF_2}_t$ can be written as $(h^{F^1}_t\ti h^{F^2}_t)$.
\end{proposition}

\subsection{Double graded bundles}
We can extend the concept of a \emph{double vector bundle} of Pradines \cite{Pradines:1974} to \emph{double graded bundles}.
However, thanks to our simple description of graded bundles in terms of associated homogeneity structures, the `diagrammatic' definition of Pradines can be substantially simplified.

As two graded bundle structure on the same manifold are described by just two homogeneity structures, the obvious concept of compatibility leads to the following (cf. Grabowski-Rotkiewicz \cite{Grabowski:2012}):
\begin{definition} A \emph{double graded bundle} is a manifold equipped with two graded bundle structures with the associated homogeneity structures $h^1,h^2$ which are \emph{compatible} in the sense that
$$h^1_t\circ h^2_s=h^2_s\circ h^1_t\quad \text {for all\ } s,t\in\R\,.$$
A double graded  bundle in which one graded  structure is that of a vector bundle is called a \emph{graded-linear bundle}, shortly a $\GrL$-bundle. In another terminology, it can be also called a $\VB$-graded bundle.
The coordinates in double graded bundles have bi-degrees composed from two degrees with  respect to the two homogeneity structures.
\end{definition}
\noindent The above condition can be also formulated as commutation of the corresponding weight vector fields, $[\nabla^1,\nabla^2]=0$.
 \begin{theorem}[Grabowski-Rotkiewicz \cite{Grabowski:2009}]
The concept of a double  vector bundle, understood as a particular double graded bundle in the above sense, coincides with that of Pradines \cite{Pradines:1974} and Mackenzie \cite{Mackenzie:1992}.
\end{theorem}

\smallskip\noindent With any  double graded bundle we can associate a commutative diagram of graded bundles and their morphisms:
$$\xymatrix{
F\ar[rr]^{h^1_0} \ar[d]^{h^2_0} && M_1\ar[d]^{h^2_0} \\
M_2\ar[rr]^{h^1_0} && M_1\cap M_2 \,.}
$$
However, this diagram does not contain full information about the double graded bundle structure. Usually, For a $\GrL$-bundle $F$ with a homogeneous structure $h$ of degree $k$ and a compatible homogeneous structure $h'$ of  degree one (vector bundle) we will write the above diagram in the form
\begin{equation}\label{bu:4} \xymatrix{
F\ar[rr]^{h_0} \ar[d]^{h'_0} && M\ar[d]^{h'_0} \\
N\ar[rr]^{h_0} && M\cap N \,,}
\end{equation}
i.e. the base of the graded bundle of degree $k$ is denoted $M$ and the base of the vector bundle structure is denoted $N$. We will often use this convention. In the $\GrL$ case, we will use bi-homogeneous local coordinates (the bi-degree is indicated below the coordinate):
\begin{equation}\label{zgc}
(\underbrace{x^{A}}_{(0,0)}, ~ \underbrace{y_{w}^{a}}_{(w,0)}, ~  \underbrace{z^{i}}_{(0,1)}, ~ \underbrace{{u}^{j}_{s}}_{(s,1)}).
\end{equation}
Here $w,s>0$.
In particular, $(x^{A},y_{w}^{a})$ are coordinates in $N$\,, $(x^{A},{z}^{i})$ are coordinates in $M$, and $(x^A)$ are coordinates in $M\cap N$.

 \begin{example}\cite[Example 5.1]{Grabowski:2012} The iterated higher tangent bundles
$$\sT^{m,n}M:=\sT^m\sT^nM\simeq\sT^n\sT^mM$$
are canonically double graded bundles.
\end{example}

The double vector bundle structures were strongly used in the Tulczyjew's approach to mechanics \cite{Tulczyjew:1974,Tulczyjew:1977}, which recently was extended to mechanics on algebroids \cite{Grabowska:2008,Grabowska:2011,Grabowska:2006}.

\begin{example}  If $E$ is a vector bundle over $M$, then $\wedge^r\sT E$  (\cite{Grabowski:2014}) is a $\GrL$-bundle. The diagram
	$$
	{\xymatrix@R-5mm @C-2mm{ & \wedge^r\sT E \ar[ld]_{\zt^r_E} \ar[rd]^{\we^r\sT\zt} & \cr
			\quad E \ar[rd] & & \wedge^r\sT M \ar[ld]\,.  \cr & M  & }}
	$$
shows a graded bundle structure ${\we^r\sT\zt}:\wedge^r\sT E\to \wedge^r\sT M$ and a vector bundle structure $\zt^r_E: \wedge^r\sT E\to E$ \cite{Grabowski:2014}.  Both structures are compatible.
	
\noindent For the case $r=2$ with the homogeneous local coordinates $(x^\zr, y^a,{\dot x}^{\zm\zn}, y^{\zs b}, z^{cd} )$ (see Example \ref{Gr}), the Euler vector field associated with the vector bundle $\zt^2_E: \wedge^2\sT E\to E$ is
	\[
	X_{\wedge^2\sT E}=\dot x^{\lambda\nu}\frac{\pa}{\pa \dot x^{\lambda\nu}}+y^{\zs b}\frac{\pa}{\pa y^{\zs b}}+\dot z^{cd}\frac{\pa}{\pa \dot z^{cd}},
	\]
	and the weight vector field associated with the graded bundle ${\we^2\sT\zt}:\wedge^2\sT E\to \wedge^2\sT M$ is the bi-tangent lift of the vector field $X_E$ to the bundle of bivectors, defined by
	\[
	d^2_{\sT}X_E=\kappa^2_M\circ \wedge ^2 \sT X_E=y^a\frac{\pa}{\pa y^a}+y^{\zs b}\frac{\pa}{\pa  y^{\zs b}}+2\dot z^{cd}\frac{\pa}{\pa \dot z^{cd}}\,,
	\]
	where the mapping  $\kappa^2_M:\sT\wedge ^2\sT M\to \wedge ^2\sT \sT M$ (for more details see \cite{Grabowski:2014}) is an isomorphism of double vector bundles.
	
	In conclusion, the coordinates $(x^\mu,y^a,\dot x^{\mu\nu},y^{\zs a},\dot z^{cd})$  are of bi-degree $(0,0)$, $(0,1)$, $(1,0)$, $(1,1)$  respectively, and the corresponding homotheties read
	\[
	\tilde h_s(x^\mu,y^a,\dot x^{\mu\nu},y^{\zs b},\dot z^{cd})=(x^\mu,y^a,s\dot x^{\mu\nu},s\cdot y^{\zs b},s\cdot\dot z^{cd})
	\]
	and
	\[
	h_t(x^\mu,y^a,\dot x^{\mu\nu},y^{\zs b},\dot z^{cd})=(x^\mu,t\cdot y^a,\dot x^{\mu\nu},t\cdot y^{\zs b},t^2\cdot \dot z^{cd})\,,
	\]
The commutativity of the above homotheties shows that $(\wedge^2\sT E,h_t)$ is indeed a weighted vector bundle of degree two. In particular, the $\GrL$-bundle $\we^2\sT \we^2\sT^\* M$,
 $$
{\xymatrix@R-5mm @C-10mm{ & \wedge^2\sT \we^2\sT^\* M \ar[ld] \ar[rd] & \cr
		\quad \we^2\sT^\* M \ar[rd] & & \wedge^2\sT M \ar[ld]  \cr & M  & }}
$$
is a GL-bundle of degree $2$.
Let $\pi_M:\sT ^*M\to M$ and $\pi^2_M: \wedge ^2 \sT^* M	\to M$ be the projections for the vector bundles $\sT ^*M$ and $\wedge ^2 \sT^* M$ onto $M$, then
$$\wedge ^2\sT \pi^2_M: \wedge ^2 \sT \wedge ^2 \sT^*M \to \wedge ^2\sT M$$
is the projection for the graded bundle $ \wedge ^2 \sT \wedge ^2 \sT^*M $ onto $\wedge ^2\sT M$ \cite{Grabowski:2014}. We can take the coordinates $(x^\mu, p_{\lambda \kappa}, \dot x^{\nu\sigma},y^{\eta}_{\theta \rho},\dot p_{\gamma \delta \epsilon\xi})$ on $\wedge ^2 \sT \wedge ^2 \sT^*M$. The Euler vector field of the vector bundle 
$$\wedge ^2 \sT \wedge ^2 \sT^*M \to \wedge ^2\sT^* M$$ is
\[
X_{\wedge ^2 \sT \wedge ^2 \sT^*M}=\dot x^{\nu\sigma}\frac{\pa}{\pa \dot x^{\nu\sigma}}+y^{\eta}_{\theta \rho}\frac{\pa}{\pa y^{\eta}_{\theta \rho}}+\dot p_{\gamma \delta \epsilon\xi}\frac{\pa}{\pa \dot p_{\gamma \delta \epsilon\xi}},
\]
and the weight vector field for the graded bundle $ \wedge ^2 \sT \wedge ^2 \sT^*M \to \wedge ^2\sT M$ is
\[
d^2_{\sT}X_{\wedge ^2 \sT^* M}=p_{\lambda \kappa} \frac{\pa}{\pa p_{\lambda \kappa}}+y^{\eta}_{\theta \rho}\frac{\pa}{\pa y^{\eta}_{\theta \rho}}+2\dot p_{\gamma \delta \epsilon\xi}\frac{\pa}{\pa \dot p_{\gamma \delta \epsilon\xi}}\,.
\]
This $\GrL$-bundle was used in \cite{Bruce:2016a,Grabowski:2014} for constructing a dynamics of strings.
\end{example}
\noindent All this can be extended to \emph{$n$-fold graded bundles} in an obvious way.
\begin{definition} A \emph{$n$-fold graded bundle} is a manifold equipped with $n$ graded bundle structures with the associated homogeneity structures $h^1,\dots,h^n$ which are \emph{compatible} in the sense that
$$h^i_t\circ h^j_s=h^j_s\circ h^i_t\quad \text {for all\ } s,t\in\R\quad \text{and}\quad i,j=1,\dots,n\,.$$
\end{definition}

\begin{proposition}
Let $(F,h^1,\dots,h^n)$ be a $n$-fold graded bundle. Then, $(F,h^{i_1}\circ\cdots\circ h^{i_k})$, where $(h^{i_1}\circ\cdots\circ h^{i_k})_t=h^{i_1}_t\circ\cdots\circ h^{i_k}_t$ is a graded bundle for all  $i_1,\dots,i_k\in \{1,\dots,n\}$, with the corresponding weight vector field
    $$\nabla=\nabla^{i_1}_F+\dots+\nabla^{i_k}_F\,.$$
\end{proposition}
\begin{remark}
All the concepts and definitions in this section apply \emph{mutatis mutandis} to $\Z$-graded bundles, so we have
\emph{double $\Z$-graded bundles}, \emph{$\Z$-graded-linear bundles ($\ZGrL$ bundles)} etc.
\end{remark}

\noindent Other natural examples of double and $n$-tuple graded bundles are obtained with the use of lifts.

\subsection{Tangent and phase lifts of homogeneity structures}
Tangent and phase lifts of homogeneity structures have been introduced in \cite[Section 2.3]{Grabowski:2013}. The tangent lifts can be generalized to higher tangent lifts (see the next section). Let $h_0:F\to M$ be a graded bundle $(F,h)$ of degree $k$ and let $x=(x^i)$ be local homogeneous coordinates in $F$. We have $h_t(x)=(t^{w_i}x^i)$ and $\nabla_F=\sum_iw_i\,x^i\,\pa_{x^i}$. The tangent bundle $\sT F$ is naturally a $\GrL$-bundle consisting of the tangent lift of the weight vector filed $\nabla _F$ and the Euler vector field of the vector bundle structure of the tangent bundle. The tangent lift of $h_t$ is $(\dt h)_t=\sT h_t$  and we have
$$
(\dt h)_t(x^i, \dot {x}^j)=(t^{w_i}x^i, t^{w_j}\dot x^j)\,.
$$
As already mentioned, the cotangent bundle $\sT^*F$ is naturally a $\Z$-graded bundle of degree $k$,
with the $\Z$-homogeneity structure $(\dt h)^*_t=(\sT h_{t^{-1}})^*$, $t\ne 0$, which in homogeneous coordinates takes the form
$$
(h^\ast)_t(x^i, p_j)=\left(t^{w_i}x^i,t^{-w_j} p_j\right)\,.
$$
According to our conventions,
$$
(h^\ast)_0(x^i, p_j)=\left(0^{w_i}x^i,0^{-w_j} p_j\right)
$$
is a projection onto $\sT^*M$. This $\Z$-graded bundle we will denote simply $\sT^*F$.

To obtain on $\sT^*F$ a structure of a graded bundle of degree $k$ we can make a procedure of shifting the weights, known from mathematical physics. To do this, we define the \emph{$k$-th phase lift} of $h_t$ as a homogeneous structure $(\dt h)^*[k]$ defined by
$$((\dt h)^*[k])_t=t^k\cdot(\sT h_{t^{-1}})^*\quad\text{for}\quad t\ne 0\,,$$
which in local coordinates looks like
\be\label{phasel}
((\dt h)^*[k])_t(x^i, p_j)=\left(t^{w_i}x^i,t^{k-w_j} p_j\right)\,.
\ee
Since all $w_i$ and $k-w_j$ are non-negative, the latter makes sense also for $t=0$ and is smooth, so we get a genuine homogeneity structure.
The graded bundle associated with this homogeneity structure we will denote $\sT^*[k]F$.
The lifts $(\dt h)$ and $(\dt h)^*[k]$, together with the obvious vector bundle structures, define $\GrL$-bundle structures on $\sT F$ and $\sT^*[k]F$ \cite[Section 2.3]{Grabowski:2013}, \cite[Example 2.17]{Bruce:2016}.
The $\Z$-graded bundle $\sT^*F$ is canonically a $\ZGrL$-bundle.
The corresponding weight vector fields are
\beas\nabla_{\sT F}&=&\sum_i\left(w_ix^i\pa_{x^i}+w_i\dot x^i\pa_{\dot x^i}\right)\,,\\
\nabla_{\sT^* F}&=&\sum_i\left(w_ix^i\pa_{x^i}-w_ip_i\pa_{p_i}\right)\,,\\
\nabla_{\sT^*[k] F}&=&\sum_i\left(w_ix^i\pa_{x^i}+(k-w_i)p_i\pa_{p_i}\right)\,.\eeas
Of course, we can  start as well with a $\Z$-graded bundle $F$ with the same formulae for the lifts. In this case, $\sT F$ is also a $\Z$-graded bundle.

\section{Duality, sections and tensor products of $\ZGrL$-bundles}
\subsection{Duality}
Let $F$ be a $\ZGrL$-bundle of degree $k$ with the associated $\Z$-homogeneous structure $h$ (see (\ref{bu:4})) and let $F^*$ will be the dual of $F$ with respect to the vector bundle structure. In the case of a double vector bundle, we should also indicate with respect to which vector bundle structure we take the duality. For the duality on $n$-tuple vector bundles we refer to \cite{Gracia:2009,Gracia:2018,Konieczna:1999,Mackenzie:2005a}.

On $F^*$ there is a canonical $\Z$-graded bundle structure of the same degree $k$, associated with a $\Z$-homogeneous structure $h^*_t=(h_{t^{-1}})^*$, $t\ne 0$. In local coordinates
\be\label{dualcoor}\left({x^{A}},{y_{w}^{a}},p_i,p_j^s\right)\,,\ee
dual to
$$\left({x^{A}},{y_{w}^{a}},{z^{i}},{u}^{j}_{s}\right)
$$
(cf. (\ref{zgc})),
$h^*_t$ takes the form
$$ h^*_t\left({x^{A}},{y_{w}^{a}},p_i,p_j^s\right)=\left({x^{A}},t^{w}{y_{w}^{a}},p_i,t^{-s}p_j^s\right)\,.$$
\noindent It is entirely obvious that  $(F^*)^*=F$. Moreover, $\sT^*F\simeq (\sT F)^*$ not only as vector bundles but as $\ZGrL$ bundles. Using the form of actions of $h_t$ and $h^*_t$, we easily get the following.
\begin{proposition}
If $\za$ and $X$ are sections of $F\to N$ and $F^*\to N$, respectively, then
$$\la h^*_t(\za),X\ran=\la\za,h_{t^{-1}}(X)\ran\,.$$
\end{proposition}

Borrowing the idea from phase lifts of graded bundles (\ref{phasel}), we can define another duality for $\GrL$-bundles $(F,h,h')$ of degree $k$; this time the dual bundle $F^*$ is again a $\GrL$-bundle.
\begin{theorem} If $F$ is a $\GrL$ bundle, then the dual bundle $F^*$ is a $\GrL$-bundle, denoted $F^*[k]$, with  the homogeneity structure $h^*[k]$. This homogeneity structure is defined by $(h^*[k])_t=t^k(h_{t^{-1}})^*$,
$$(h^*[k])_t\left({x^{A}},{y_{w}^{a}},p_i,p_j^s\right)=\left({x^{A}},t^w{y_{w}^{a}},t^kp_i,t^{k-s}p_j^s\right)\,,$$
where local coordinates are as in (\ref{dualcoor}).
\end{theorem}
\begin{proof}
It is easy to see that $(h^*[k])_t\circ (h^*[k])_{t'}=(h^*[k])_{tt'}$. Moreover, coordinates ${x^{A}},{y_{w}^{a}},p_i,p_j^s$ are of degrees $0,w,k,k-s$, respectively, and all these degrees are $\ge 0$.

\end{proof}
\begin{remark} Note that in general, the degree of $F^*[k]$ is only $\le k$. For instance, if with respect to the graded bundle structure of degree $k$, $F$ has only coordinates of degree $2,k$, $k>2$, then $F^*$ has coordinates of degree $(-2,-k)$, and $F^*[k]$ has coordinates of degree $(k-2,0)$, so is of degree $k-2$. However, we still have $(F^*[k])^*[k]=F$.
\end{remark}
\subsection{The degree of sections}
Let $F$ be a $\ZGrL$-bundle (\ref{bu:4}) with bi-homogeneous coordinates (\ref{zgc}).
\begin{definition} We say that a section $\zs:N\to F$,
$$\zs(x^A,{y_{w}^{a}})=\left(x^A,{y_{w}^{a}},z^i(x^A,{y_{w}^{a}}),{u}^{j}_{s}(x^A,{y_{w}^{a}})\right)$$
of the vector bundle structure \emph{is of degree $\zl\in\R$} if
\be\label{section} h_t\left(\zs(h_{t^{-1}}(x^A,{y_{w}^{a}}))\right)=t^{-\zl}\zs(x^A,{y_{w}^{a}})\ee
for $t> 0$.
\end{definition}
\begin{example}
The vector field $\pa_{x^i}$ on a $\Z$-graded bundle $F$ is of degree $-w_i$ as a section of the $\ZGrL$-bundle $\sT F$.
\end{example}
\noindent Note that if $F$ is a $\GrL$-bundle, the degrees of sections can be only integer numbers.

\medskip\noindent
For a section $\zs$ of the vector bundle structure $\zt:F\to N$, we denote with $\zi(\zs)$ the linear function on $F^*$ which reads
$\zi(\zs)(e^*_x)=\la\zs(x),e^*_x\ran$. In local homogeneous coordinates $(x,y)$ on $F$ and the dual coordinates $(x,p)$ on $F^*$, for $\zs(x)=(x^i,\zs^a(x))$, we have
$$\zi(\zs)(x,p)=\sum_ap_a\cdot\zs^a(x)\,.$$
The section $\zs$ is uniquely determined by the submanifold $\zs(N)$ of $F$. Conversely, any submanifold $S$ of $F$ which is mapped diffeomorphically on $N$ by the vector bundle projection $h'_0:F\to N$ is the image of a section. For $t\ne 0$ we denote by $h_t(\zs)$ the section $\zs'$ of $F\to N$ corresponding to the submanifold $h_t(\zs(N))$.
\begin{theorem}\label{sections1}
Suppose $F$ is a $\ZGrL$-bundle. Then, a section $\zs:N\to F$ of the vector bundle structure is of degree $\zl\in\Z$ if and only if $\zi(\zs)$ is a function of degree $\zl$ on the $\ZGrL$-bundle $F^*$. This is equivalent to the identity $h_t(\zs(x))=t^{-\zl}\zs(h_t(x))$, i.e.
$$h_t(\zs)=t^{-\zl}\zs.$$
\end{theorem}
\begin{proof}
One can directly compute that, for $t\ne 0$,
\be\label{homsection} h_t\left(\zs(h_{t^{-1}}(x^A,{y_{w}^{a}}))\right)=\left(x^A,{y_{w}^{a}},z^i\circ h_{t^{-1}}(x^A,{y_{w}^{a}}),t^s\cdot {u}^{j}_{s}\circ h_{t^{-1}}(x^A,{y_{w}^{a}})\right)\,.\ee
Equality (\ref{section}) means
$$z^i\circ h_{t^{-1}}(x^A,{y_{w}^{a}})=t^{-\zl}z^i(x^A,{y_{w}^{a}})\quad\text{and}\quad
t^s\cdot {u}^{j}_{s}\circ h_{t^{-1}}(x^A,{y_{w}^{a}})=t^{-\zl}{u}^{j}_{s}(x^A,{y_{w}^{a}})\,.$$
This is equivalent to the statement that $z^i(x^A,{y_{w}^{a}})$ and ${u}^{j}_{s}(x^A,{y_{w}^{a}})$ are functions on $N$ of degrees $\zl$ and $\zl+s$, respectively. But this in turn is equivalent to the fact that
$$\zi(\zs)\left({x^{A}},{y_{w}^{a}},p_i,p_j^s\right)=\sum_ip_i\cdot z^i(x^A,{y_{w}^{a}})+\sum_jp_j^s\cdot{u}^{j}_{s}(x^A,{y_{w}^{a}})$$
is of degree $\zl$ on $F^*$.
Further,
$$h_t(\zs)(x^A,{y_{w}^{a}})=\left(x^A,{y_{w}^{a}},z^i\circ h_{t^{-1}}(x^A,{y_{w}^{a}}),t^s\cdot {u}^{j}_{s}\circ h_{t^{-1}}(x^A,{y_{w}^{a}})\right)$$
which is exactly the right hand of (\ref{homsection}) and leads to the same degree of homogeneity of $\zs$.

\end{proof}
\begin{definition}
We say that a linear map $\zF:\Sec(F_1)\to\Sec(F_2)$ between the vector bundle sections of $\ZGrL$-bundles $F_1$ and $F_2$ is of degree $\zl$ if for a vector bundle section $\zs$ of $F_1$ with degree $w$, $\zF(\zs)$ is a vector bundle section of $F_2$ with degree $\zl+w$
\end{definition}

\begin{example}
Let $F=N\ti V$ be a $\ZGrL$-bundle with the trivial vector bundle structure. Then, the $\Z$-homogeneity structure $h$ on $F$ splits into the product of $\Z$-graded bundles with $\Z$-homogeneity structures $h^N$ and $h^V$ on $N$ and $V$, respectively. A section $\zs:N\to F$ is of degree $\zl$ if and only if the corresponding map $\zs^V:N\to V$ between graded bundles is of degree $\zl$.
\end{example}

\begin{example}
Let $F$ be a $\Z$-graded bundle and $\zW^l(F)$ be the space of differential $l$-forms on $F$ as  sections of $\we^l\sT^*F$. Then, de Rham differential
$$\xd:\zW^l(F)\to\zW^{l+1}(F)$$
is of degree $0$.
\end{example}
\begin{example}
Let $X$ be a vector field of degree $\zl$ on the graded bundle $F$. Then the  contraction
$$i_X:\zW^l(F)\to\zW^{l-1}(F)$$
is of degree $\zl$.
\end{example}

\noindent Using Theorem \ref{sections1}, we easily get the following Theorem.
\begin{theorem}
Suppose $F$ is a $\ZGrL$-bundle of degree $k$ with a $\Z$-homogeneity structure $h$.
\begin{itemize}
\item Then, a section $\zs:N\to F$ of the vector bundle structure is of degree $\zl\in\R$ if and only if $\zi(\zs)$ is a function of degree $\zl+k$ on the $\ZGrL$-bundle $F^*[k]$.
\item The pairing $\la\cdot,\cdot\ran:F\ti_N F^*[k]\to \R$, where $\R$ is a graded bundle with the trivial homogeneity structure, is a map of degree $k$.
\end{itemize}
\end{theorem}

\subsection{Tensor products}
Affine coordinates $(x^i,y^j)$ on a vector bundle $E\to M$ are associated with local coordinates on $M$ and a local basis $\{ e_j\}$ of sections of $E$. The correspondence between the basis and linear coordinates $(y^j)$ is given by
$$y^j\left(\sum_la^l\,e_l(x)\right)=a^j\,.$$ This works also for the vector bundle structure of a $\ZGrL$-bundle $F$.
As we can take the coordinates $y^l$ bi-homogeneous of the graded degree $w_l$, the sections $e_l$ are homogeneous of degree $-w_l$. Indeed, for $t\ne 0$,
\beas t^{w_j}a^j&=&t^{w_j}y^j\left(\sum_la^le_l(x)\right)=y^j\circ h_t\left(\sum_la^le_l(x)\right)\\
&=&y^j_s\left(\sum_la^l\left(h_t(e_l(x))\right)\right)=
y^j\left(\sum_la^l\,h_t(e_l)(h_t(x))\right)\,.\eeas
This implies that $h_t(e_l)=t^{w_l}(e_l)$, thus $e_l$ is of degree $-w_l$  (Theorem \ref{sections1}).

Now consider two $\ZGrL$-bundles: $F_1$ of degree $k_1$ and $F_2$ of degree $k_2$. Let us assume that both vector bundles are over the same manifold $N$, and that the restrictions of $h^1_t$ and $h^2_t$ to $N$ are equal. The bases of the $\Z$-graded bundle structure may be different, $M_1$ and $M_2$, respectively. Let $(x^i_s,y^j_s)$ be affine coordinates on the vector bundle $F_s\to N$, associated with a local basis of section $\{ e^s_l\}$ of $F_s\to N_s$, $s=1,2$.

\medskip
Consider the tensor product $F_1\ot_NF_2$ of these vector bundles. We can take a local basis of sections of this tensor product of the form $\{ e^1_j\ot e^2_l\}$, and the corresponding linear coordinates in $F_1\ot_NF_2$ we will denote $y^j_1\ot y^l_2$. Put $h^\ot_t:F_1\ot_NF_2\to F_1\ot_NF_2$ of the form
$$h^\ot_t(e^1_j(x)\ot e^2_l(x))=\left(h_t^1(e^1_j(x))\ot h_t^2(e^2_l(x))\right)\,,\quad t\in\R\,.$$
The tensor product on the right hand side makes sense, as both vectors $h_t^1(e^1_j(x))$ and $h_t^2(e^2_l(x))$ have the same initial point $h^1_t(x)=h^2_t(x)$. It is easy to see that $h^\ot_t$ is a linear map, so it is compatible with the vector bundle structure on the tensor product and an action of the monoid $(\R,\cdot)$. This means that $F_1\ot_NF_2$ is a $\ZGrL$-bundle. To see the degrees of coordinates $y_1^a\ot y^b_2$, consider
\beas &(y^a_1\ot y_2^b)\circ h^\ot_t\left(\sum_{j,l} d^{jl}(e^1_j(x)\ot e^2_l(x))\right)=(y^a_1\ot y_2^b)
\left(\sum_{j,l} d^{jl}\left(h^1_t(e^1_j(x))\ot h^2_t(e^2_l(x))\right)\right)=\\
& (y^a_1\ot y_2^b)\left(\sum_{j,l} d^{jl}\left(t^{w_j^1}e^1_j(h^1_t(x))\right)\ot \left(t^{w_l^2}e^2_l(h^2_t(x))\right)\right)=t^{w_a^1+w_b^2}d^{ab}\,,\eeas
which shows that
$$\deg(y_1^a\ot y_2^b)=w_a^1+w_b^2$$
and sections $e^1_a\ot e^2_b$ are of degree $-(w_a^1+w_b^2)$.
The degree of $x^i$ is the same as the degree of $x^i$ on $F_1$ (or $F_2$).
In particular, $F_1\ot_NF_2$ is a graded bundle if $F_1$ and $F_2$ are graded bundles and its degree is $\le k_1+k_2$.
\begin{example}
Let $F$ be a $\Z$-graded bundle of degree $k$ over $M$ with homogeneous coordinates $(x^i)$, so that $\sT^*F$ and $\sT F$ are canonically $\ZGrL$-bundles with homogeneous coordinates $(x^i,\dot x^j)$ and $(x^i,p_j)$. Note that $\dot x^j$ is of degree $w_j$ and $p_j$ is of degree $-w_j$, so $f(x)\pa_{x^j}$ is of degree $\deg(f)-w_j$ and $\deg(f(x)\xd x^i)$ is $\deg(f)+w_j$. All tensor products of $\sT F$ and $\sT^*F$ are $\ZGrL$-bundles. One can easily check that the degree of a tensor field $K$ on $F$ coincides with the degree of $K$ viewed as a section of the corresponding tensor product of $\ZGrL$-bundles $\sT^*F$ and $\sT F$.
\end{example}

Suppose we have a $q$-contravariant and $p$-covariant tensor field $K$ on $F$, $K\in\mathscr{T}^q_p(F)$, of degree $\zl_K$ and of the form
$$K=f_K(x)\pa_{x^{i_1}}\ot\cdots\ot\pa_{x^{i_q}}\ot\xd x^{j_1}\ot\cdots\ot\xd x^{j_p}\,.$$
Let us take contravariant and covariant tensor fields $X\in\mathscr{T}^l_0(F)$ and $\zw\in\mathscr{T}^0_u(F)$, $u\le q$ and $l\le p$, of degrees $\zl_X$ and $\zl_\zw$.
$$X=f_X(x)\pa_{x^{a_1}}\ot\cdots\ot\pa_{x^{a_l}}\quad\text{and}\quad\zw=f_\zw(x)\xd x^{b_1}\ot\cdots\ot\xd x^{b_u}\,.$$
We define the insertion maps $i_XK$ and $i_\zw K$ as follows:
$$i_XK=f_K(x)f_X(x)\zd^{a_1}_{j_1}\cdots\zd^{a_l}_{j_l}\pa_{x^{i_1}}\ot\cdots\ot\pa_{x^{i_q}}\ot\xd x^{j_{l+1}}\ot\cdots\ot\xd x^{j_p}$$
and
$$i_\zw K=f_K(x)f_\zw(x)\zd^{b_1}_{i_1}\cdots\zd^{b_u}_{i_u}\pa_{x^{i_{u+1}}}\ot\cdots\ot\pa_{x^{i_q}}\ot\xd x^{j_{1}}\ot\cdots\ot\xd x^{j_p}\,.$$
This defines linear maps $i_X:\mathscr{T}^q_p(F)\to\mathscr{T}^q_{p-l}(F)$ and $i_\zw:\mathscr{T}^q_p(F)\to\mathscr{T}^{q-u}_{p}(F)$.
\begin{proposition}
The map $i_X$ is of degree $\zl_X$ and $i_\zw$ is of degree $\zl_\zw$.
\end{proposition}
\begin{proof}
We calculate $\deg(i_XK)-\deg(K)$ and $\deg(i_\zw K)-\deg(K)$:
\beas \deg(i_XK)-\deg(K)&=&\deg(f_X)-w_{j_{1}}-\cdots -w_{j_l}=\deg(X)\,,\\
\deg(i_\zw K)-\deg(K)&=&\deg(f_\zw)+w_{i_{1}}+\cdots+w_{i_u}=\deg(\zw)\,.
\eeas

\end{proof}
\noindent Since symmetrization or skew-symmetrization do not change the degree of a tensor, the above proposition is valid also
for symmetric or anti-symmetric tensors.

In the following, we will focus our attention on graded bundles, although most of the concepts and results can be formulated for $\ZGrL$-bundles as well.

\section{Higher lifts of tensor fields and distributions}
In this section, we will try to understand the compatibility of a homogeneity structure with other geometric structures, such as a general tensor or a distribution.
\begin{definition} Structures on a graded bundle $F$ which are compatible with the homogeneity structure we will call \emph{weighted structures}, e.g. weighted Poisson structures or weighted Nijenhuis tensors. If $F$ is a vector bundle (graded bundle of degree 1), then weighted structures on $F$ we will call $\VB$-structures. Indeed, $\VB$-groupoids and $\VB$-algebroids (\cite{Bursztyn:2016,Drummond:2015,Gracia:2010,Gracia:2017,Mackenzie:1998,Mackenzie:2000,Mackenzie:2005a, Mackenzie:2011,Pradines:1988}) are natural examples of $\VB$-structures in our sense.
\end{definition}

As one would expect, the main question is the meaning of compatibility.
Instead of proposing an \emph{ad hoc} definition, we will try to make an intelligent guess taking as examples
tensor fields that are canonical in some sense.
An example of a canonical homogeneity structure is the one on higher tangent bundles  $$\zt^r_M:\sT^rM=\sJ^r_0(\R,M)\to M\,.$$ There is a huge list of various concepts of lifting tensor fields and other geometric structures from $M$ to $\sT^rM$ (e.g. \cite{Gancarzewicz:1990,Grabowski:2014,Grabowski:2012,Kolar:1996a,Wamba:2012,Morimoto:1969,Morimoto:1970,Morimoto:1970a, Popescu:2010,Salimov:2019,Yano:1968}), starting from the complete tangent lifts  \cite{Grabowski:1995,Grabowski:1997,Ishihara:1973,Wamba:2014,Wamba:2011,Ozkan:2012,Yano:1966,Yano:1966a, Yano:1967,Yano:1967a,Yano:1973}.
We will use mainly \cite{Morimoto:1970c} and \cite{Wamba:2012}, where the descriptions of lifts (prolongations) are the same although based on different concepts.

Our assumption is that the complete lifts of tensor fields and distributions from $M$ to $\sT^rM$ should form structures compatible with the canonical homogeneity structure on $\sT^rM$.
Let us fix a non-negative integer $r$ for the rest of the section. We will construct lifts of tensors from a manifold $M$ to $\sT^rM$.
\begin{definition}[\cite{Wamba:2012,Morimoto:1970c}] Let $f\in C^\infty(M)$ and $\zl$ be a non-negative integer not bigger than $r$. Then, $\zl$-lift of $f$ is the function $L_\zl(f)=f^{(\zl)}$ on $\sT^rM$ defined by
$$f^{(\zl)}([\phi]_r)=\frac{1}{\zl!}\left[\frac{\xd^{\zl}(f\circ\phi)}{\xd\,t^{\zl}}\right]_{t=0}\,,$$
for $[\phi]_r\in\sT^rM$, where $\phi:\R\to M$ is a smooth curve. We put by convention $f^{(\zl)}=0$ for $\zl<0$.
\end{definition}
One can see \cite{Morimoto:1970c} that $\zl$-lifting $L_\zl:C^\infty(M)\to C^\infty(\sT^rM)$ is linear and generalized Leibniz rule
$$(f\cdot g)^{(\zl)}=\sum_{\zm=0}^\zl f^{(\zm)}\cdot g^{(\zl-\zm)}$$ is satisfied
for all $f,g\in C^\infty(M)$. Moreover, for local coordinates $x^1,\dots,x^n$ on $M$ we have $(x^i)^{(\zl)}=x^i_\zl$, where $(x^i,x^j_\zn)$, $\zn=1,\dots,r$, are the induced coordinates on $\sT^rM$.
The $\zl$-lifts of one-forms $\zw\in\zW^1(M)$ and vector fields $X\in\mathfrak{X}(M)$ are defined as follows.
\begin{theorem}\label{t6}
\
\begin{itemize}
\item
There exists one and only one $\R$-linear lift $L_\zl:\zW^1(M)\to\zW^1(\sT^rM)$ such that
$$L_\zl(f\cdot \rmd g)=(f\cdot \rmd g)^{(\zl)}:=\sum_{\zm=0}^\zl f^{(\zm)}\rmd g^{(\zl-\zm)}\,.$$
In particular, $(\rmd x^i)^{(\zl)}=\rmd x^i_\zl$.
\item There exists one and only one $\R$-linear lift $L_\zl:\mathfrak{X}(M)\to\mathfrak{X}(\sT^rM)$ such that
for $L_\zl(X)=X^{(\zl)}$ we have
$$X^{(\zl)}f^{(\zm)}=(Xf)^{(\zl+\zm-r)}\,.$$
In particular, $(\pa_{x^i})^{(\zl)}=\pa_{x^i_{r-\zl}}$.
\end{itemize}
\end{theorem}
The lifts $f^{(r)}$, $\zw^{(r)}$, and $X^{(r)}$ will be called \emph{complete lifts to $\sT^rM$} and denoted also
$f^{(c)}$, $\zw^{(c)}$, and $X^{(c)}$.

\begin{remark}
By convention, $f^{(\zl)}=0$, $\zw^{(\zl)}=0$, and $X^{(\zl)}=0$ if $\zl<0$ or $\zl>r$.
\end{remark}
\begin{remark}
If the vector field $X\in\mathfrak{X}(M)$ induces a one-parameter group of transformations $\psi_t$, then $X^{c}\in\mathfrak{X}(\sT^rM)$ induces the one-parameter group of transformations $\sT^r\psi_t$ on $\sT^rM$. In the case of the tangent bundle $\sT M$, i.e. in the case $r=1$, the 0-lift of $X$ is identical with the vertical lift of $X$, while the 1-lift of $X$ is identical with the complete lift of $X$, as defined in \cite{Grabowski:1995,Yano:1973}.

\begin{example} For $r=1$, in natural coordinates $(x^i,\dot x^j)$ the above lifts of a vector field $X=X^i\frac{\partial}{\partial x^i}$ read
$$X^{(0)}=X^i \frac{\partial}{\partial \dot x^i},\qquad X^{(1)}=X^i\frac{\partial}{\partial x^i}+\frac{\partial X^k}{\partial x^j}\dot x^j\frac{\partial}{\partial \dot x^k}.$$
For $r=2$ in coordinates $\left(x^i_0,x^j_{1},x^k_{2}\right)$ on $\sT^2M$ we get
$$\begin{array}{l}\vspace{5pt}
\displaystyle X^{(0)}=X^i \frac{\partial}{\partial x^i_{2}}\,, \\ \vspace{5pt}
\displaystyle X^{(1)}=X^i\frac{\partial}{\partial x^i_{1}}+\frac{\partial X^k}{\partial x^j} \,x^j_{1}\,\frac{\partial}{\partial x^k_{2}}\,,\\
\displaystyle X^{(2)}=X^i\frac{\partial}{\partial x^i}+ \frac{\partial X^k}{\partial x^j}\, x^j_{1}\,\frac{\partial}{\partial x^k_{1}}+
        \left(\frac{1}{2}\frac{\partial^2 X^l}{\partial x^n\partial x^m}\,x^n_{1}\, x^m_{1}+\frac{\partial X^l}{\partial x^p}\, x^p_{2}\right)\frac{\partial}{\partial x^l_{2}}.
\end{array}$$
For a one-form $\alpha=\alpha_i\rmd x^i$ we have in turn
$$\begin{array}{l}\vspace{5pt}
\displaystyle \alpha^{(0)}=\alpha_i\,\rmd x^i; \\ \vspace{5pt}
\displaystyle \alpha^{(1)}=\frac{\partial \alpha_i}{\partial x^j}\, x^j_{1}\,\rmd x^i+\alpha_k\,\rmd x^k_{1};\\
\displaystyle \alpha^{(2)}=\left(\frac{1}{2}\frac{\partial \alpha_i}{\partial x^k\partial x^j}\, x^k_{1}\, x^j_{1}+\frac{\partial \alpha_i}{\partial x^l}\, x^l_{2}\right)\,\rmd x^i+
                 \frac{\partial \alpha_m}{\partial x^n}\, x^n_{1}\,\rmd x^m_{1} + \alpha_p\,\rmd x^p_{2}.
\end{array}$$
\end{example}

In \cite{Wamba:2012} the authors define the same lifts of functions, one-forms, and vector fields by means of canonical isomorphisms (see \cite{Cantrijn:1989,Grabowski:1995,Grabowski:1999,Kolar:1996})
$$\zk^r_M:\sT^r\sT M\to\sT\sT^rM\,,\quad \ze^r_M:\sT^r\sT^*M\to\sT^*\sT^rM\,.$$
The lifts of one-forms and vector fields are defined as
$$\zw^{(\zb)}=\ze^r_M\circ\zq^{(r-\zb)}_{\sT^*M}\circ\sT^r\zw\,,\quad X^{(\zb)}=\zk^r_M\circ\zq^{(\zb)}_{\sT M}\circ\sT^rX\,.$$
Here, we view $\zw$ and $X$ as sections $\zw:M\to\sT^*M$ and $X:M\to\sT M$. Here, for a vector bundle $E\to M$ $\zq^{(\zb)}_E$ is a map $\zq^{(\zb)}_E:\sT^rE\to\sT^rE$ defined by
$$\zq^{(\zb)}_E(j_0^r\phi)=j^r_0(t^\zb\phi)\,.$$Recall that for a manifold $N$,  the map $\zq^{(\zb)}_N:\sT^rN\to\sT^rN$ is defined by
$$\zq^{(\zb)}_N(j_0^r\phi)=j^r_0(t^\zb\phi)\,,$$
where $\phi$ is a curve in $E$.
We get the same lifts as in \cite{Morimoto:1970c} with one exception: $X^{(\zl)}$ in \cite{Wamba:2012} is the same as
$X^{(r-\zl)}$ in \cite{Morimoto:1970c}. We will be using the notation of \cite{Morimoto:1970c}.
\end{remark}

\begin{theorem}(Morimoto \cite{Morimoto:1970c})
\
\begin{itemize}
\item If $X\in\mathfrak{X}(M)$ and $f\in C^\infty(M)$, then
\be\label{e0}(f\cdot X)^{(\zl)}=\sum_{\zm=0}^\zl f^{(\zm)}\,X^{(\zl-\zm)}\,.\ee
In particular,
$$\left(\sum_ia_i\pa_{x^i}\right)^{(\zl)}=\sum_i\sum_{\zn=r-\zl}^ra_i^{(\zn+\zl-r)}\pa_{x^i_\zn}\,.$$
\item If $X,Y\in\mathfrak{X}(M)$, then \be\label{Lb} [X^{(\zl)},Y^{(\zm)}]=[X,Y]^{(\zl+\zm-r)}\,.\ee
\item If $X\in\mathfrak{X}(M)$ and $\zw\in\zW^1(M)$, then
$$ i_{X^{(\zl)}}\zw^{(\zm)}=(i_X\zw)^{(\zl+\zm-r)}\,.$$
\end{itemize}
\end{theorem}
\noindent Finally, we apply the generalized Leibniz rule for the lifts of tensor product:
\be\label{e2}(T\otimes S)^{(\zl)}=\sum_{\zm=0}^\zl T^{(\zm)}\ot S^{(\zl-\zm)}\ee
to obtain the lifts of multivector fields
$$ (X_1\we\dots\we X_q)^{(\zl)}=\sum_{\zm_1+\dots+\zm_q=\zl}(X_1)^{(\zm_1)}\we\dots\we(X_q)^{(\zm_q)}$$
and differential forms
$$(\za_1\we\dots\we\za_p)^{(\zl)}=\sum_{\zm_1+\dots+\zm_p=\zl}(\za_1)^{(\zm_1)}\we\dots\we(\za_p)^{(\zm_p)}\,.$$
Actually, we can obtain this way the lifts of arbitrary $q$-contravariant and $p$-covariant tensor fields:
$$L_\zl:\mathscr{T}^q_p(M)\to\mathscr{T}^q_p(\sT^rM)\,,$$ so for arbitrary tensor fields we have the lift
$$L_\zl:\mathscr{T}(M)\to\mathscr{T}(\sT^rM)\,,$$ where $\mathscr{T}(M)=\oplus_{p,q}\mathscr{T}^q_p(M)$. The lifts $L_r(K)$ we will call \emph{complete lifts} and denote with $K^{(c)}$.

\begin{theorem}[\cite{Morimoto:1970c}]
\
\begin{itemize}
\item If $\zw$ is a $p$-form on $M$, then
$$\rmd \zw^{(\zl)}=(\rmd\zw)^{(\zl)}\,,$$
and $$i_{X^{(\zl)}}\zw^{(\zm)}=(i_X\zw)^{(\zl+\zm-r)}\,.$$
\item If $X\in\mathfrak{X}(M)$ and $K\in\mathscr{T}(M)$, then
$$\Ll_{X^{(\zl)}}K^{(\zm)}=(\Ll_X K)^{(\zl+\zm-r)}\,.$$
\end{itemize}
\end{theorem}
One can generalize (\ref{Lb}) to the Schouten (Schouten-Nijenhuis) bracket of multivector fields (see \cite{Grabowski:2013a,Grabowski:1995,Wamba:2012,Schouten:1940}) which is a graded bracket of degree $-1$ on the graded space of multivector fields.
Recall that the Schouten bracket on multivector fields takes  the form
\be\label{Sch} [X_1\wedge\dots \wedge X_k, Y_1\wedge \dots \wedge
Y_l]_S = \sum_{i,j}(-1)^{i+j} [X_i,Y_j] \wedge X_1\wedge
\dots\wedge \widehat{X}_i \wedge \dots \wedge X_k\wedge
Y_1\wedge \dots \wedge \widehat{Y}_j\wedge \dots \wedge Y_l\,.\ee
This formula together with (\ref{Lb}) gives the following.
\begin{theorem}
If $X$ and $Y$ are multivector fields on $M$, then the Schouten bracket $[\cdot,\cdot]_{S}$ is related to
the lifts by
$$ [X^{(\zl)},Y^{(\zm)}]_{S}=[X,Y]_S^{(\zl+\zm-r)}\,.$$
\end{theorem}
\begin{corollary}[\cite{Grabowski:1995,Wamba:2012}]\label{wP}
The complete lift preserves the Schouten bracket
$$[X^{(c)},Y^{(c)}]_{S}=[X,Y]_{S}^{(c)}\,.$$
In particular, the complete lift of a Poisson tensor is a Poisson tensor.
\end{corollary}
For vector valued forms from $\zW(M,\sT M)$, we have the Nijenhuis-Richardson bracket \cite{Grabowski:2013a,Grabowski:1997,Kolar:1996,Lecomte:1992}:
If $\zm\in\zW^k(M)$ and $\zn\in\zW^l(M)$ and $X,Y\in\mathfrak{X}(M)$, then
$$[\zm\ot X,\zn\ot Y]_{NR}=\zm\we i_X\zn\ot Y+(-1)^ki_Y\zm\we\zn\ot X\,.$$
\begin{theorem}[\cite{Lecomte:1992}]
 The Nijenhuis-Richardson bracket makes the space of vector valued forms $\zW(M,\sT M)$ into a graded Lie algebra. The
graded bracket is of degree $-1$.
\end{theorem}

\begin{theorem} For vector valued forms $\zm\ot X$ and $\zn\ot Y$ we have
$$[(\zm\ot X)^{(\zl)},(\zn\ot Y)^{(u)}]_{NR}=[\zm\ot X,\zn\ot Y]_{NR}^{(\zl+u-r)}\,.$$
In particular, the complete lift preserves the Nijenhuis-Richardson bracket
$$[(\zm\ot X)^{(c)},(\zn\ot Y)^{(c)}]_{NR}=[\zm\ot X,\zn\ot Y]_{NR}^{(c)}\,.$$
\end{theorem}
\begin{proof}
\beas & [(\zm\ot X)^{(\zl)},(\zn\ot Y)^{(u)}]_{NR}=
\sum_{\za=0}^\zl \sum_{\zb=0}^u[\zm^{(\za)}\ot X^{(\zl-\za)},\zn^{(\zb)}\ot Y^{(u-\zb)}]_{NR}=\\
&\sum_{\za=0}^\zl \sum_{\zb=0}^u\left(\zm^{(\za)}\we i_{X^{(\zl-\za)}}\zn^{(\zb)}\ot Y^{(u-\zb)}+
(-1)^k(i_{Y^{(u-\zb)}}\zm^{(\za)}\we\zn^{(\zb)}\ot X^{(\zl-\za)}\right)=\\
&\sum_{\za=0}^\zl \sum_{\zb=0}^u\left(\zm^{(\za)}\we (i_X\zn)^{(\zb+\zl-\za-r)}\ot Y^{(u-\zb)}+
(-1)^k\left((i_Y\zm)^{(u-\zb+\za-r)}\we\zm^{(\za)}\ot X^{(\zl-\za)}\right)\right)=\\
&=[\zm\ot X,\zn\ot Y]_{NR}^{(\zl+u-r)}\,.
\eeas
\end{proof}
\noindent There is another interesting bracket on the space of vector valued forms, namely the  \emph{Fr\"olicher-Nijenhuis bracket} \cite{Frolicher:1956,Grabowski:2013a,Grabowski:1997,Kolar:1996}.
The {Fr\"olicher-Nijenhuis bracket} is
defined for simple tensors $\zm \otimes X$ and
$\zn \otimes Y$, where $X,Y\in\mathfrak{X}(M)$, $\zm\in\zW^k(M)$ and $\zn\in\zW^l(M)$, by the formula
                \bea\label{FN1} &
        [\zm \otimes X, \zn \otimes Y]_{FN} =
        \zm\wedge \zn \otimes
[X,Y] + \zm\wedge \Ll_X\zn\otimes Y -\Ll_Y\zm \wedge \zn \otimes
X + \\ \nn &(-1)^k(\xd\zm \wedge i_X\zn \otimes Y + i_Y\zm
\wedge \xd\zn \otimes X)\,.
                \eea
\begin{theorem}[\cite{Frolicher:1956}]
        The formula (\ref{FN1}) defines a graded Lie bracket
$[\cdot,\cdot]_{FN}$ of degree $0$ on the graded space of vector valued forms $\zW(M,\sT M)$.
\end{theorem}

\begin{theorem}\label{t1} For vector valued forms $\zm\ot X$ and $\zn\ot Y$, where $\zm$ is a $k$-form, we have
$$[(\zm\ot X)^{(\zl)},(\zn\ot Y)^{(u)}]_{FN}=[\zm\ot X,\zn\ot Y]_{FN}^{(\zl+u-r)}\,.$$
In particular, the complete lift preserves the Fr\"olicher-Nijenhuis bracket
$$[(\zm\ot X)^{(c)},(\zn\ot Y)^{(c)}]_{FN}=[\zm\ot X,\zn\ot Y]_{FN}^{(c)}\,.$$
\end{theorem}
\begin{proof}
\beas & [(\zm\ot X)^{(\zl)},(\zn\ot Y)^{(u)}]_{FN}=
\sum_{\za=0}^\zl \sum_{\zb=0}^u[\zm^{(\za)}\ot X^{(\zl-\za)},\zn^{(\zb)}\ot Y^{(u-\zb)}]_{FN}=\\
&\sum_{\za=0}^\zl \sum_{\zb=0}^u\zm^{(\za)}\we\zn^{(\zb)}\ot[X^{(\zl-\za)},Y^{(u-\zb)}]+\\
&\sum_{\za=0}^\zl \sum_{\zb=0}^u\left(\zm^{(\za)}\we \Ll_{X^{(\zl-\za)}}\zn^{(\zb)}\ot Y^{(u-\zb)}
-\Ll_{Y^{(u-\zb)}}\zm^{(\za)}\we\zn^{(\zb)}\ot X^{(\zl-\za)}\right)+\\
&(-1)^k\sum_{\za=0}^\zl \sum_{\zb=0}^u\left(\rmd(\zm^{(\za)})\we i_{X^{(\zl-\za)}}\zn^{(\zb)}\ot Y^{(u-\zb)}+
i_{Y^{(u-\zb)}}\zm^{(\za)}\we\rmd(\zn^{(\zb)})\ot X^{(\zl-\za)}\right)=\\
&\sum_{\za=0}^\zl \sum_{\zb=0}^u\zm^{(\za)}\we\zn^{(\zb)}\ot[X,Y]^{(\zl-\za+u-\zb-r)}]+\\
&\sum_{\za=0}^\zl \sum_{\zb=0}^u\left(\zm^{(\za)}\we (\Ll_X\zn)^{(\zb+\zl-\za-r)}\ot Y^{(u-\zb)}-
(\Ll_Y\zm)^{(u-\zb+\za-r)}\we\zn^{(\zb)}\ot X^{(\zl-\za)}\right)+\\
&(-1)^k\sum_{\za=0}^\zl \sum_{\zb=0}^u\left((\rmd\zm)^{(\za)}\we (i_{X}\zn)^{(\zl-\za+\zb-r)}\ot Y^{(u-\zb)}+
(i_{Y}\zm)^{(u-\zb+\za-r)}\we(\rmd\zn)^{(\zb)}\ot X^{(\zl-\za)}\right)=\\
&=[\zm\ot X,\zn\ot Y]_{FN}^{(\zl+u-r)}\,.
\eeas
\end{proof}
\noindent Let $N$ be a $(1,1)$-tensor on $M$. We can also consider $N$ as a linear map $N:\sT M\to\sT M$. Such maps can be composed, so we can define $N_1\circ N_2$.
\begin{theorem}[\cite{Morimoto:1970c}]\label{t2}
The complete lifts to $\sT^rM$ preserve the composition of $(1,1)$ tensors
$$(N_1\circ N_2)^{(c)}=N_1^{(c)}\circ N_2^{(c)}\,.$$
Moreover, the complete lift of the identity map $I_{\sT M}:\sT M\to\sT M$ is the identity map $I_{\sT\sT^rM}:\sT\sT^rM\to\sT\sT^rM$.
\end{theorem}
\begin{corollary}
Complete lifts of almost complex structures are almost complex structures.
\end{corollary}
\noindent Let us recall that a \emph{Nijenhuis tensor} on $M$ is a $(1,1)$-tensor such that
$$[N,N]_{FN}=0\,.$$
From Theorem \ref{t1} and Theorem \ref{t2} we get immediately the following.
\begin{corollary}[\cite{Morimoto:1970c}] The complete lift of a Nijenhuis tensor $N$ to $\sT^rM$ is a Nijenhuis tensor.
The complete lift of a complex structure is a complex structure.
\end{corollary}
\noindent Now, we will check the degrees of complete lifts of tensors with respect to the canonical homogeneity structure on $\sT^rM$. This homogeneity structure has the weight vector fields which in adapted coordinate system $(x^i,x^i_\zm)$ in $\sT^rM$ reads
$$\nabla_{\sT^rM}=\sum_{i=1}^n\sum_{\zm=1}^r\zm x^i_\zm\pa_{x^i_\zm}\,.$$
It is easy to see that the lifts of functions $f^{(\zl)}$ have the weight $\zl$,
$$\Ll_{\nabla_{\sT^rM}}f^{(\zl)}=\zl\cdot f^{(\zl)}\,.$$
Then, it is easy to calculate the weights of lifts of general tensor fields.
\begin{theorem}\label{t3} The degrees of the lifts of tensors to $\sT^rM$ are the following:
\begin{description}
\item{a)} For any differential form $\zw=\za_1\ot\cdots\ot\za_p$ on $M$, we have
$$\wu(\zw^{(\zl)})=\zl\,.$$
\item{b)} For any $q$-vector field $X=X_1\ot\dots\ot X_q$ on $M$, we have
$$\wu(X^{(\zl)})=\zl-qr\,.$$
\item{c)} For any $(q,p)$-tensor $K=\za_1\ot\dots\ot\za_p\ot X_1\ot\dots\ot X_q$ in $M$, we have
$$\wu(K^{(\zl)})=\zl-qr\,.$$
\item{d)} Permutations of factors in a homogeneous tensor product do not change the degree. Thus the degrees of wedge products are the same as tensor products.
\end{description}

In particular,
$$\wu\left((\zw_1^{(c)}\we\dots\we\zw_p^{(c)}\ot X_1^{(c)}\we\dots\we X_q)^{(c)}\right)=-(q-1)r$$
and complete lifts of vector fields commute with $\nabla_{\sT^rM}$.
\end{theorem}
\begin{proof}
The homogeneity structure $h$ on $\sT^rM$ may be defined by $h_s([\phi]_r)=[\phi_s]_r$, where $\phi_s(t)=\phi(st)$.
For any $f\in C^\infty(M)$ we have then
$$f^{(\zl)}(h_s([\phi]_r))=f^{(\zl)}([\phi_s]_r)=\frac{1}{\zl!}\left[\frac{\xd^{\zl}(f(\phi(st)))}{\xd\,t^{\zl}}\right]_{t=0}
=\frac{s^\zl}{\zl!}\left[\frac{\xd^{\zl}(f(\phi(t)))}{\xd\,t^{\zl}}\right]_{t=0}=s^\zl f^{(\zl)}([\phi]_r)\,,
$$
that means that $f^{(\zl)}$ is of degree $\zl$. Moreover, we already know that $(x^i)^{(\zl)}=x^i_\zl$.
From Theorem \ref{t6} it easily follows that for a one-form $\za$ and a vector field $X$
on $M$ the degrees of $\za^{(\zl)}$ and $X^{(\zl)}$ are, respectively, $\zl$ and $\zl-r$. In particular,
$(\xd x^i)^{(\zl)}=\xd x^i_\zl$ and $\pa_{x^i}=\pa_{x^i_{r-\zl}}$.
The rest follows directly from the rule (\ref{e2}) of lifting tensor products.

\end{proof}
\begin{theorem}\label{tx}
The complete lift $\nabla_F^{(c)}$ of a homogeneity structure $\nabla_F$ on the graded bundle $F$ of degree $k$ is a homogeneity structure of degree $k$ on $\sT^rF$ compatible with the canonical homogeneity structure on $\sT^rF$,
$$[\nabla_{\sT^rF},\nabla_F^{(c)}]=0\,.$$
This shows that the higher tangent bundle $\sT^rF$ of a graded bundle $\zt:F\to M$ is canonically a double graded bundle of bi-degree $(r,k)$:
$$\xymatrix{
\sT^rF\ar[rr]^{\zt^r_F} \ar[d]^{\sT^r\zt} && F\ar[d]^{{\zt}} \\
\sT^rM\ar[rr]^{\zt^r_M} && M }\,.
$$
\end{theorem}
\begin{proof}
The vector field $\nabla_F^{(c)}$ is a weight vector field of the homogeneity structure $\sT^r(h_t)$, where $h_t$ is the homogeneity structure on $F$. The maps $\sT^r(h_t)$ define indeed a homogeneity structure on $\sT^rF$, as
$$\sT^r(h_t)\circ\sT^r(h_s)=\sT^r(h_t\circ h_s)=\sT^r(h_{ts})\,.$$
According to Theorem \ref{t3}, $\nabla_F^{(c)}$ commutes with $\nabla_{\sT^rF}$, so the two homogeneity structures on ${\sT^rF}$ are compatible. In local homogeneous coordinates $(x^i,x^j_\zm)$ in $\sT^rF$,
$$\nabla_F^{(c)}=\left(\sum_{i=1}^nw_ix^i\pa_{x^i}\right)^{(c)}=\sum_{\zm=1}^r\sum_{i=1}^nw_i\cdot x^i_\zm\,\pa_{x^i_\zm}$$
and
$$\nabla_{\sT^rF}=\sum_{\zm=1}^r\sum_{i=1}^n\zm\cdot x^i_\zm\,\pa_{x^i_\zm}\,,$$
so the lifted homogeneity structure is of degree $k$ with respect to $\nabla_F^{(c)}$ and of degree $r$ with respect to $\nabla_{\sT^rF}$.
\end{proof}
\section{Weighted structures}
To fix our attention, we concentrate in this section on graded bundles only, although most of the concepts and results work for $\Z$-graded bundles as well.

Roughly speaking, \emph{weighted structures} are geometric structures on graded bundles which are compatible with the homogeneity structure. What the compatibility means, we will make precise for a list of geometric structures using their higher lifts as natural examples. If the graded bundle is a vector bundle, the corresponding weighted structures we will call \emph{$\VB$-structures}. This concept of $\VB$-structures coincides with the already known  in the literature $\VB$-structures for Lie algebroids and Lie groupoids.
\subsection{Weighted tensor fields}
Motivated by the resuts of the previous section, we propose the following.
\begin{definition}\label{wt}
Let $K$ be a $(q,p)$-tensor field on a graded bundle $F\to M$ of degree $k$ with the weight vector field $\nabla_F$.
We call the tensor field $K$ \emph{compatible} with the homogeneity structure on $F$ if $\deg(K)=-(q-1)k$, i.e.
$$\Ll_{\nabla_F}(K)=-(q-1)k\cdot K\,.$$ In this case we call the structure $(F,\nabla_F,K)$ a \emph{weighted $K$-structure}.
\end{definition}
\noindent In particular, we get the following.
\begin{definition}
\
\begin{itemize}
\item A \emph{weighted Nijenhuis} manifold is a graded bundle $F$ equipped with a Nijenhuis tensor of degree $0$.
\item A \emph{weighted almost complex manifold} is a graded bundle $F$ equipped with a weighted almost complex structure, i.e. a $(1,1)$-tensor $N$ of degree $0$ such that $N\circ N=-I_{\sT F}$.
\item An \emph{weighted almost product manifold} is a graded bundle $F$ equipped with a weighted product structure, i.e. a $(1,1)$-tensor $N$ of degree $0$ such that $N\circ N=I_{\sT F}$.
\item An \emph{weighted almost tangent manifold} is a graded bundle $F$ equipped with a weighted tangent structure, i.e. a $(1,1)$-tensor $N$ of degree $0$ such that $N\circ N=0$.
\end{itemize}
We get weighted complex (resp., product, tangent) structures if $N$ is a Nijenhuis tensor.
\end{definition}
\begin{proposition}
If $A$ and $B$ are weighted multivector fields on a graded bundle $F$ of degree $k$, then the Schouten bracket $[A,B]_S$ is also weighted.
\end{proposition}
\begin{proof}
Suppose $A=X_1\wedge\dots \wedge X_l$ is weighted of degree $(1-l)k$ and
$B=Y_1\wedge \dots \wedge Y_m$ is of degree $(1-m)k$. As for vector fields $X,Y,Z$ on a  manifold, we have
$$\Ll_Z[X,Y]=[Z,[X,Y]]=[[Z,X],Y]+[[X,[Z,Y]]=[\Ll_Z(X),Y]+[X,[\Ll_Z(Y)]]\,,$$
by Definition  \ref{Sch} of the Schouten bracket,
\beas &\Ll_{\nabla_F}([A,B]_S)=\\
&\Ll_{\nabla_F}\left(\sum_{i,j}(-1)^{i+j} [X_i,Y_j] \wedge X_1\wedge
\dots\wedge \widehat{X}_i \wedge \dots \wedge X_l\wedge
Y_1\wedge \dots \wedge \widehat{Y}_j\wedge \dots \wedge Y_m\right)=\\
&[\Ll_{\nabla_F}(A),B]_S+[A,\Ll_{\nabla_F}(B)]=(1-l+1-m)k\,[A,B]_S\,.
\eeas
The multivector field $[A,B]_S$ is a $(l+m-1)$-vector field, so it is weighted if and only if it is of degree
$(1-(l+m-1))k$, and we have just shown that it is true.

\end{proof}
\begin{proposition}
Let $\zm\in\zW^m(F)$ and $\zn\in\zW^l(F)$ be differential forms on a graded bundle $F$ of degree $k$ and let $A=\zm\ot X$ and $B=\zn\ot Y$ be weighted vector valued differential forms, i.e. they are of degree 0. Then, the Fr\"olicher-Nijenhuis bracket $[A,B]_{FN}$ is also weighted.
\end{proposition}
\begin{proof}
According to (\ref{FN1}),
\beas &\Ll_{\nabla_F}([A,B]_{FN})=\\
& \Ll_{\nabla_F}\left(\zm\wedge \zn \otimes[X,Y] + \zm\wedge \Ll_X\zn\otimes Y -\Ll_Y\zm \wedge \zn \otimes
X\right)+ \\ &(-1)^m\Ll_{\nabla_{F}}\left(\xd\zm \wedge i_X\zn \otimes Y + i_Y\zm\wedge \xd\zn \otimes X\right)\,.
                \eeas
With the use of the identities
\beas \Ll_{\nabla_F}\Ll_X\zn&=&\Ll_X\Ll_{\nabla_F}\zn+\Ll_{(\Ll_{\nabla_F}(X))}\zn\,,\\
\Ll_{\nabla_F}\xd\zn&=&\xd\Ll_{\nabla_F}\zn\,,\\
  \Ll_{\nabla_F}i_X\zn&=&i_X\Ll_{\nabla_F}\zn+i_{(\Ll_{\nabla_F}(X))}\zn\,,
\eeas
and direct calculations, we get
$$\Ll_{\nabla_F}([A,B]_{FN})=[\Ll_{\nabla_F}(A),B]_{FN}+[A,\Ll_{\nabla_F}(B)]_{FN}\,.$$
Since $A$ and $B$ are of degree $0$, then  $[A,B]_{FN}$ is also of degree $0$, thus weighted.

\end{proof}
\begin{proposition}
Let $\zm\in\zW^m(F)$ and $\zn\in\zW^l(F)$ be differential forms on a graded bundle $F$ of degree $k$ and let $A=\zm\ot X$ and $\zn\ot Y$ be weighted vector-valued differential forms. Then the Nijenhuis-Richardson bracket $[A,B]_{NR}$ is also weighted.
\end{proposition}
\begin{proof}
The proof is completely analogous to the proof for the Fr\"olicher-Nijenhuis bracket. We have
\[
\Ll_{\nabla_f}([\zm\ot X,\zn\ot Y]_{NR})=\Ll_{\nabla_f}(\zm\we i_X\zn\ot Y+(-1)^ki_Y\zm\we\zn\ot X)\,.
\]	
By direct calculations we get
$$\Ll_{\nabla_F}([A,B]_{NR})=[\Ll_{\nabla_F}(A),B]_{NR}+[A,\Ll_{\nabla_F}(B)]_{NR}\,.$$
Since $A$ and $B$ are of degree $0$,  $[A,B]_{FN}$ is then of degree $0$, thus weighted.
	
\end{proof}
\subsection{Weighted vector bundles and distributions}

\begin{definition} A \emph{weighted vector bundle of degree $k$} is a vector bundle $E\to M$ equipped additionally
	with a homogeneity structure $h$ of degree $k$ such that $h_t:E\to E$ are vector bundle morphisms for all $t\in\R$. In particular, it means that  $N=h_0(E)$ is a vector  subbundle of $E\to M$. We denote a weighted vector bundle with the couple $(E,h_t)$.
\end{definition}
\begin{proposition}[\cite{Grabowski:2009}] For a homogeneity structure on a vector bundle $E\to M$, the maps  $h_t:E\to E$ are vector bundle morphisms for all $t\in\R$ if and only if $h$ commutes with the homogeneity structure $\tilde h$ defining the vector bundle structure:
	$$h_t\circ \tilde h_s=\tilde h_s\circ h_t\,,$$
	for all $t,s\in\R$.
\end{proposition}
The above proposition shows that weighted vector bundles are just $\GrL$-bundles.

\begin{definition}\

\begin{itemize}

\item A (smooth) distribution $D\subset\sT F$ on a graded bundle $(F,h)$ covering a submanifold $F_0\subset F$, is a \emph{weighted distribution} if it is a graded subbundle of the tangent bundle $\sT F$ with the lifted homogeneity structure  $\dt h$, i.e.
$$\sT(h_t)(D)\subset D\quad\text{for all}\quad t\in\R\,.$$
\item A \emph{weighted foliation} is a weighted distribution which is involutive.
\item A \emph{weighted fibration} is a fibration $\zt:F\to N$ such that the vertical foliation (foliation into fibers of $\zt$) is weighted.
\end{itemize}
\end{definition}
\begin{theorem}\label{t5}
\
\begin{itemize}
\item Let $h_0:F\to M$ be a graded bundle and let $\zt_F:\sT F\to F$ be the canonical projection. Assume additionally that  $D\subset\sT F$ is a weighted distribution  covering a submanifold $F_0$ of $F$, $\zt_F(D)=F_0$. Then,
\be\label{t4a} h_t\circ \zt_F=\zt_F\circ \sT h_t\ee
 and $F_0$ is a graded subbundle of $F$.
\item Let $\mathcal{F}$ be a foliation on a graded bundle $h_0:F\to M$. Then, $\mathcal{F}$ is weighted if and only if $h_t$ maps leaves of $\mathcal{F}$ into leaves, for all $t\in\mathbb R$.
\item Suppose that the fibration $\zt:F\to N$ is weighted. Then, the  homogeneity structure $h$ on $F$ induces a homogeneity structure $\zf$ on $N$ such that $\zf_t\circ\zt=\zt\circ h_t$, $N_0=\zf_0(N)$ is a submanifold in $N$
    and the restriction of $\zt$ to $M=h_0(F)$ gives a fibration $\zt:M\to N_0$. That is, $N$ is canonically a graded bundle and $M$ is canonically a fiber bundle.

\end{itemize}
\end{theorem}
\begin{proof}
\
\begin{itemize}
\item Let us take $v_p\in D_p$, $p\in F$. As $D$ is invariant with respect to $\sT h_t$, $\sT h_t(v_p)$ belongs to $D$ and $\zt_F(\sT h_t(v_p))=h_t(p)$ that is equivalent to (\ref{t4a}). Of course, (\ref{t4a}) implies trivially that $F_0$ is invariant with respect to $h_t$, i.e. it must be a graded subbundle of $F$.
\item Suppose $\cF$ is weighted.  Let $D=\sT\mathcal{F}$ be the corresponding involutive distribution and let $\zt_F:\sT F\to F$ be the canonical projection. Let us take a smooth curve $\zg:\R\to F$ which completely belongs to one leaf, say $\cO$, i.e. $\zg(s)\in\cO$ for all $s$.
Let $\dot\zg:\R\to\sT\cO$ be the tangent prolongation of $\zg$. Of course, $\zt_F(\sT h_t(\dot\zg(s)))=h_t(\zg(s))$
and, as $\cF$ is weighted, $\sT h_t(D)\subset D$, so $\sT h_t((\dot\zg(s))\in D_{h_t(\zg(s))}$. This implies that the curve $s\mapsto(\sT h_t(\dot\zg(s)))\in\sT F$ lies in $\sT\cF$. But a curve on $F$ whose tangent vectors at each point are tangent to leaves of $\cF$ must belong to one leaf.

Conversely, if $h_t$ maps leaves into leaves, then $\sT h_t$ maps vectors of $D=\sT\cF$ into vectors belonging to $\sT\cF$; the distribution $D$ is weighted and so the foliation is weighted.
\item First, note that $N$ is not a submanifold of $F$. However, as diffeomorphism $h_t$ maps fibers into fibers, it induces a smooth map $\zf_t:N\to N$ such that $h_t(\cF_x)\subset\cF_{\zf_t(x)}$. Here $\cF_x$ is the fiber of $\zt$ over the point $x\in N$. It is easy to see that $\zf_t\circ\zf_s=\zf_{ts}$, so that $\zf$ is a homogeneity structure on $N$ over $\zf_0(N)=N_0$, induced by $h$,  and $\zf_t\circ\zt=\zt\circ h_t$, for all $t\in \R$.
We conclude also that $\zt$ restricted to $M$ is a smooth surjection onto $N_0$. Indeed,
$$\zt(M)=\zt\circ h_0(F)=\zf_0\circ\zt(F)=\zf_0(N)=N_0\,.$$
The appropriate diagram looks as follows.
\begin{equation}\label{cu:4} \xymatrix{
F\ar[rr]^{h_0} \ar[d]^{\zt} && M\ar[d]^{\zt_{|M}} \\
N\ar[rr]^{\zf_0} && N_0. }\,
\end{equation}
Passing to local trivializations of $\zt$ we can assume that $\zt:F\to N$ is trivial, i.e. $F=N\ti\cF_0$. As
$h_t:F\to F$ maps fibers into fibers, it induces a homogeneity structure $\bar h_t$ on $\cF_0$ with $M_0=\bar h_0(\cF_0)$ as the base and
$$h_t(p,q)=(\zf_t(p),\bar h_t(q)).$$
Here $(p,q)\in N\ti\cF_0=F$.
The diagram (\ref{cu:4}) takes the form
$$\xymatrix{
F=N\ti\cF_0\ar[rr]^{h_0} \ar[d]^{\zt} && M=N_0\ti M_0\ar[d]^{\zt_{|M}} \\
N\ar[rr]^{\zf_0} && N_0. }\,
$$
Because
$$M=h_0(F)=h_0(N\ti\cF_0)=(\zf_0(N)\ti\bar h_0(\cF_0))=N_0\ti M_0$$
and $\zt_{|M}:N_0\ti M_0\to N_0$ is the obvious projection, $\zt_{|M}:M\to N_0$ is a fibration.
\end{itemize}

\end{proof}

\begin{proposition}\label{p4} (cf. \cite{Wamba:2012})
Let $D$ be a distribution of rank $k$ on a manifold $M$ and $D^{(r)}\subset\sT\sT^rM$ be a generalized distribution generated by all vector fields $X^{(\zl)}$, $\zl=0,\dots,r$, for vector fields $X$ on $M$ which belong to $D$. Then, $D^{(r)}$ is a weighted distribution of rank $(r+1)k$ on $\sT^rM$. If the distribution $D$ is involutive, then the distribution $D^{(r)}$ is involutive as well.
\end{proposition}
\begin{lemma}\label{lema}
Let $X$ be a vector field on $M$,  $X(x)\ne 0$. Then, the vectors $\{ X^{(0)}(y),\dots,X^{(r)}(y)\}$ are linearly independent at each $y\in\sT^rM$ which projects to $x$ under the canonical projection $\zt_M^r:\sT^rM\to M$.
In particular, all vectors $X^{(\zn)}(y)$ are different from $0$.
Moreover, the natural projection $\zt^r_\zm:\sT^rM\to\sT^\zm M$, where $\zm\le r$, projects $X^{(\zl)}$ to zero
if only $\zm<r-\zl$, and to the vector field  $X^{(\zm)}_{\sT^{\zm}M}$ on $\sT^{\zm}M$ if $\zm=r-\zl$,  where $X^{(\zm)}_{\sT^{\zm}M}$ is the complete lift of $X$ to $\sT^{\zm}M$.
\end{lemma}
\begin{proof}
It is well known that for any vector field $X$ on $M$ which does not vanish at $x\in M$ there is a neighborhood $U$ of $x$ and coordinates $(x^i)$ on $U$ in which $X$ is rectified, i.e. takes the form $X=\pa_{x^1}$. Then, according to Theorem \ref{t6}, in the induced coordinates on $\sT^rM$ we have $X^{(\zl)}=\pa_{x^1_{(r-\zl)}}$, $\zl=0,\dots,r$ and the proof is complete.
\end{proof}
\begin{proof}[Proof of Proposition \ref{p4}]
Assume that $D$ is, locally in $U\subset M$, generated by vector fields $X_j$ on $M$, $j=1,\dots,k$,  linearly independent at each $x\in U$. Hence, any vector field $X$ belonging to $D$ is locally of the form $\sum_jf_jX_j$. But, according to Theorem (\ref{e0}),
$$(f_j\cdot X_j)^{(\zl)}=\sum_{\zm=0}^\zl f_j^{(\zm)}\,X_j^{(\zl-\zm)}$$
is again a combination of vector fields
\be\label{e6} X_j^{(\zn)}\,,\ \zn=0,\dots,r,\quad\text{and}\quad j=1,\dots,k\,,\ee
which implies that $D^{(r)}$ is locally generated in $(\zt_M^r)^{-1}(U)$ by vector fields (\ref{e6}), where $\zt^r_M:\sT^rM\to M$ is the graded bundle projection, so $D^{(r)}$ has rank $\le k(r+1)$.
On the other hand, the vector fields (\ref{e6}) are linearly independent at each point $y$ of $(\zt_M^r)^{-1}(U)$.
Indeed, suppose that
$$X(y)=\sum_{\zn,j}a^j_\zn X_j^{(\zn)}(y)=0\,,\quad\text{where}\quad a^j_\zn\in\R\,.$$
Let $\zn_0$ be the highest $\zn$ for which at least one of $a^j_\zn$ is non-zero.
Then, according to Lemma \ref{lema}, $\zt^r_{\zn_0}:\sT^rM\to\sT^{\zn_0} M$ projects $X(y)$ to
$$\sum_{j}a^j_{\zn_0} X_j^{(\zn_0)}(y)=0\,.$$
But
$$\sum_{j}a^j_{\zn_0} X_j^{(\zn_0)}(y)=\left(\sum_{j}a^j_{\zn_0} X_j\right)^{(\zn_0)}(y)$$
which, again by Lemma \ref{lema} is zero only if $\sum_{j}a^j_{\zn_0} X_j(x)=0$ for $x=\zt^r_M(y)\in U$.
Since $X_j(x)$ are linearly independent, all $a^j_{\zn_0}$ are $0$. This contradicts the choice of $\zn_0$.
Hence, the vector fields (\ref{e6}) are linearly independent at each point $y$ of $(\zt_M^r)^{-1}(U)$, so the rank of $D^{(r)}$ is exactly $k(r+1)$.

As $D^{(r)}$ are generated by vector fields (\ref{e6}), the distribution $\sT h_t(D^{(r)})$ is spanned by
$$\Big\{\sT h_t\left(X_j^{(\zn)}\right)\Big\}=\{t^{-\zn}X_j^{\zn}\,|\ \zn=0,\dots,r,\quad\text{and}\quad j=1,\dots,k\}\,,$$
so equals $D^{(r)}$. Here, $h_t$ is the canonical homogeneous structure on $\sT^rM$.

\medskip\noindent
If $D$ is involutive, then $[X,Y]\in D$ if $X,Y\in D$. But then, according to (\ref{Lb}),
$$[X^{(\zl)},Y^{(\zm)}]=[X,Y]^{(\zl+\zm-r)}\,,$$ so $D^{(r)}$ is involutive.

\end{proof}
\subsection{Weighted Ehresmann connections}
Let $\zt:F\to N$ be a weighted fibration on a graded bundle $h_0:F\to M$, i.e. $h_t$ maps fibres of $\zt$ onto fibres of $\zt$ for all $t\in\R$. Denote with $\sV F$ the vertical distribution of $\zt$ corresponding to this foliation. The distribution $\sV F$ is weighted.
An Ehresmann connection on $F$ is a distribution $D\subset\sT F$ complementary to $\sV F\subset\sT F$. In other words,
$\sT F=\sV F\oplus D$.
\begin{definition}
The Ehresmann connection $D$ is \emph{weighted} if $D$ is a weighted distribution.
\end{definition}
\begin{example}
Let $\tau: E\rightarrow M$ be a vector bundle, i.e. a graded bundle of degree $1$. A linear connection in a vector bundle is usually introduced in a form of a covariant derivative $\nabla:\Sec(\tau_M)\times \Sec(\tau)\rightarrow \Sec(\tau)$ being linear with respect to the first factor, and a first order derivation with respect to the second factor. Equivalently, a linear connection in a vector bundle can be defined as an Ehresmann connection in the fibration $\tau$ such that the horizontal distribution $D$ is a double vector subbundle of the double vector bundle $\sT E$. This means that $D$ is a weighted distribution.
\end{example}

\begin{example}\label{conn:1}
Let $\tau: E\rightarrow M$ be a vector bundle with linear connection $D$. If $E$ is of dimension $n+m$, where $m$ is the dimension of $M$, then the distribution $D$ is of rank $m$. Let $D^{(r)}$ denote the lift of $D$ to $\sT^r E$. According to the proposition \ref{p4} the distribution $D^{(r)}$ is of rank $(r+1)m$ which is equal to the dimension of the manifold $\sT^rM$. The horizontal vector fields that span $D$ in $\tau^{-1}(U)$ for some domain of coordinates $U\subset M$, can be chosen in the following form
$$X_k=\partial_{x^k}-\Gamma^A_{kB}(x)y^B\partial_{y^A}, \quad k\in{1,\ldots,m}$$
where $(x^i, y^A)$ are coordinates on $E$, with weight $0$ for $(x^i)$ and weight $1$ for $(y^B)$. Coordinates $(x^i)$ are, as usual, pull-backs of coordinates on the base manifold $M$. Functions $\Gamma^A_{kB}$ depend on base coordinates only and constitute the Christoffel symbols of the connection. The distribution $D^{(r)}$ is spanned by all the lifts $X_k^{(\lambda)}$ for $\lambda\in\{0,\ldots,r\}$. Since
$$\sT\sT^r\tau(X_k^{(\lambda)})=\partial_{x^k_{(r-\lambda)}}\,,$$
then
$$\sT\sT^r\tau(D^{(k)})=\sT\sT^r M.$$
We have then the splitting $\sT\sT^rE=\sV\sT^rE\oplus_{\sT^rE}D^{(r)}$ where vertical vectors are vertical with respect to the r-tangent projection $\sT^r\tau:\sT^rE\to\sT^r M$. The distribution $D^{(r)}$ is a weighted Ehresmann connection on the canonical graded bundle $\sT^rE$. Moreover, this connection is a linear connection in the vector bundle $\sT^r\tau$, since horizontal vector fields $X^{(\lambda)}_k$ are of the form
\be\label{conn:2}X^{(\lambda)}_k=\partial_{x^k_{(r-\lambda)}}-\sum_{\mu=0}^{\lambda}\sum_{\nu=0}^{\mu}
(\Gamma^A_{kB})^{\mu-\nu}y^B_{(\nu)}\partial_{y^A_{(r-\lambda+\mu)}},\ee
i.e. coefficients are linear with respect to coordinates $y^A_{(\rho)}$ for $\rho\in\{0,\ldots, r\}$.

For every linear Ehresmann connection in a vector bundle, there is a covariant derivative defined on sections of the bundle. It is enough to give the covariant derivative of the basis elements $e_{(\nu),B}$ of the space of sections such that an element of $\sT^r E$ over a point in $\sT^r M$ can be written as $\sum_{B=1}^{n}\sum_{\nu=0}^{r} y^B_{(\nu)}e_{(\nu)B}$. The covariant derivative of the section $e_{(\mu)B}$ reads
$$\nabla_{\partial_{x^k_{(\lambda)}}}e_{(\mu),B}=\sum_{\rho=\lambda}^{r}(\Gamma^A_{kB})^{(\rho-\lambda-\mu)}e_{(\rho),A}.$$
The above formula shows that the Christoffel symbols of the lifted connections are the lifts of the Christoffel symbols of the original connection. Here, we adopt the convention that $f^{(\lambda)}=0$ if $\lambda<0$.
\end{example}
\begin{example} In \cite{Morimoto:1970c} Morimoto defined the complete lift of an affine connection on $M$ to $\sT^r M$ as the only affine connection on $\sT^r M$ such that the covariant derivative of the lifted connection $\nabla^{(r)}$ satisfies
$$\nabla^{(r)}_{X^{(\lambda)}}Y^{(\mu)}=(\nabla_XY)^{\lambda+\mu-r}$$
for every pair $X,Y$ of vector fields on $M$. An affine connection on a manifold $M$ is a specific example of a linear connection in a vector bundle, namely a linear connection in the tangent bundle $\tau_M:\sT M\rightarrow M$, that can be equivalently described as a double vector subbundle $D$ of $\sT\sT M$. Applying the lift  from Example \ref{conn:1}, we get the distribution $D^{(r)}$ of $\sT^r\sT M$ which is a double-weighted subbundle of $\sT\sT^r\sT M$ with respect to the projections $\tau_{\sT^r \sT M}$ and $\sT\sT^r\tau_M$.
We claim that the covariant derivative $\nabla^{(r)}$ is associated to the linear connection in the bundle $\tau_{\sT^r M}:\sT\sT^r M\rightarrow \sT^r M$ with the horizontal distribution
$$\sT\sT\sT^r M\supset\mathcal{D}=\sT\kappa^r_M(D^{(r)}),$$
where
$$\kappa^r_M: \sT^r\sT M\rightarrow\sT\sT^r M$$
is the canonical isomorphism. Starting from the coordinate system $(x^i)$ on $M$ we construct the adopted coordinate system $(x^i, \delta x^j)$ for $\sT M$ and then the lifted coordinate system
$(x^i, \delta x^j, x^i_{1}, \delta x^j_{1}, \ldots, x^i_{r}, \delta x^j_{r})$ for $\sT^r\sT M$. On the other hand, we can start from $(x^i)$  on $M$ to get
$(x^i, x^i_{1}, \ldots, x^i_{r})$ {\cm for} $\sT^r M$ and then $(x^i, x^i_{1}, \ldots, x^i_{r},\delta x^i, \delta x^i_{1}, \ldots, \delta  x^i_{r})$. This does not lead
to any confusion, since the canonical isomorphism $\kappa_M^r$ in these coordinates is expressed as an appropriate permutation:
$$(x^i, x^i_{1}, \ldots, x^i_{r},\delta x^i, \delta x^i_{1}, \ldots, \delta  x^i_{r})\circ \kappa_M^r=
(x^i, \delta x^j, x^i_{1}, \delta x^j_{1}, \ldots, x^i_{r}, \delta x^j_{r})\,.$$
Horizontal vector fields of the lifted distribution $D^{(r)}$ on $\sT^r\sT M$ we get as in (\ref{conn:2}),
\be\label{conn:3}X^{(\lambda)}_k=\partial_{x^k_{r-\lambda}}-\sum_{\mu=0}^{\lambda}\sum_{\nu=0}^{\mu}
(\Gamma^i_{kj})^{\mu-\nu}\delta x^j_{\nu} \partial_{\delta x^i_{r-\lambda+\mu}}.\ee
Horizontal vector fields on $\sT\sT^r M$ spanning $\mathcal{D}$ are obtained from $X^{(\lambda)}_k$ by push-forward with $\kappa^r_M$, therefore they look in coordinates exactly as (\ref{conn:3}). Consequently, the Christofell symbols of the connection associated to $\mathcal{D}$ are precisely
$$\Gamma^{(i,\rho)}_{(k,\mu)(j,\nu)}=(\Gamma^{i}_{kj})^{(\rho-\mu-\nu)}\,,$$
as in formula (5.4) of \cite{Morimoto:1970c}.
\end{example}

\subsection{Weighted Poisson, symplectic and pseudo-Riemannian structures}
According to Definition \ref{wt}, a Poisson tensor $\zL$ is compatible with a graded bundle structure on $\zt:F\to M$
of degree $k$ if it has degree $-k$. We deal then with a \emph{weighted Poisson structure of degree $k$}. In particular, if a manifold $M$ is equipped with a Poisson structure $\zL$, then the tangent lift $\dt\zL$ (see \cite{Grabowski:1995}) is a weighted Poisson structure on the tangent bundle $\sT M$. This is exactly the Poisson tensor that defines the Lie algebroid structure on $\sT^*M$ associated with $\zL$ \cite{Grabowski:1979,Grabowski:1995}. Actually, all this can be generalized to higher complete lifts of Poisson structures $\zL$ on $M$, which define weighted Poisson structures $\zL^{(c)}$ of degree $r$ on $\sT^rM$ (cf. Theorem \ref{wP}), and to arbitrary $2$-contravariant tensors. Manifolds equipped with a $2$-contravariant tensors $\zL$ are called in \cite{Grabowski:1999} \emph{Leibniz structures (Leibniz manifolds)}. Leibniz structures define the Leibniz brackets of functions,
$$\{ f,g\}_\zL=\zL(\rmd f,\rmd g)\,.$$ A smooth map $\zvf:M\to N$ between Leibniz manifolds is a \emph{morphism of Leibniz structures} if it relates the two Leibniz tensors (or Leibniz brackets). The definition of \emph{weighted Leibniz structures of degree $k$} is the same as  in the case of Poisson structures, i.e. the Leibniz tensor should be of degree $-k$.

\begin{proposition}\label{p2}
A weighted Leibniz structure of degree $k$ is a Leibniz manifold $(F,\zL)$ equipped additionally with a homogeneity structure $h$ of degree $k$, such that the Leibniz bracket of any two homogeneous functions $f_{w_1},g_{w_2}\in C^\infty(F)$ of degree $w_1$ and $w_2$, respectively, is a homogeneous function of degree $w_1+w_2-k$,
$$\deg\left(\{ f_{w_1},g_{w_2}\}_\zL\right)=w_1+w_2-k\,.
$$
This is equivalent to the fact that the morphism
$$\zL^\sharp :\sT^*[k] F\to \sT F\,,\quad \za\mapsto i_\za\zL\,,$$
is a morphism of $\GrL$-bundles.
\end{proposition}
\begin{proof}
	In a local system of homogeneous coordinates $(x_1,...,x_n)$ with weights respectively $w_1,\dots,w_n$, we have
	\[
	\zL=\sum_{i,j} \zL^{ij}\frac{\pa}{\partial{x^i}}\ot\frac{\pa}{\partial{x^j}}\,,
\]
where $\zL^{ij}=\left\langle \zL, dx^i\ot dx^j\right\rangle =\{x^i,x^j\}_\zL$  are smooth functions on $F$.
The tensor $\zL$ is of degree $-k$ if and only if $\deg(\zL^{ij})=\deg(\{ x^i,x^j\}_\zL)=w_i+w_j-k$. As any homogeneous function is locally a polynomial in coordinates  $(x_1,...,x_n)$, it is also true for arbitrary homogeneous functions $f_{w_1},g_{w_2}\in C^\infty(F)$.

The map $\zL^\sharp$ in the adapted coordinates $(x^i,\dot x^j)$ and $(x^i,p_j)$ on $\sT F$ and $\sT^*F$ reads
$$(x^i,\dot x^j)\circ \zL^\sharp=\left(x^i,\sum_lp_l\zL^{lj}\right)\,.$$ As coming from a tensor, it is obviously linear with respect to the vector bundle structures on $\sT F$ and $\sT^*F$, and as $\deg(p_l)=k-w_l$ for the phase lift of the homogeneity structure on $F$, it is clearly also a morphism of the lifted graded bundle structures on $\sT^*F$ and $\sT F$, i.e. $\deg(\sum_l\zL^{lj}p_l)=w_j$.

\end{proof}
\noindent The map $\zL^\sharp :\sT ^* F\to \sT F$ can serve also for characterization of the fact that $\zL$ is a Poisson bivector \cite{Grabowski:1979a}.
\begin{example}
Since the complete lift $\zL^{(r)}$ of a weighted Poisson (Leibniz) tensor $\zL$ on a graded bundle $F$ of degree $k$ to $\sT^rF$ is a Poisson (Leibniz) tensor of degree $-r$, the higher tangent bundle $\sT^rM$ of a Poisson (Leibniz) manifold $F$ is canonically a weighted Poisson (Leibniz) structure with respect to the canonical graded bundle structure on $\sT^rF$. But $\sT^rF$ is a double graded bundle with the second homogeneity structure, here of degree $k$, being the complete lift of the homogeneity structure $h$ on $F$ (see Theorem \ref{tx}). The corresponding weight vector field is the complete lift $\nabla^{(r)}$ of the weight vector field of $h$. Note that $\zL^{(r)}$ is also a weighted Poisson tensor with respect to $\nabla^{(r)}$. Indeed, according to Corollary \ref{wP},
$$[\nabla^{(r)},\zL^{(r)}]_S=[\nabla,\zL]_S^{(r)}=(-k\zL)^{(r)}=-k\zL^{(r)}\,.$$
One can say therefore that the higher tangent bundle $\sT^rF$ of a weighted Poisson manifold $(F,h,\zL)$ of degree $k$ is a double-weighted Poisson manifold of bi-degree $(r,k)$.
\end{example}
Let $E\to N$ be a vector bundle. Similarly like a linear Poisson tensor on $E^*$ induces a Lie algebroid bracket on $E$, any linear Leibniz tensor $\zL$ on $E^*$ induces a bracket $[X,Y]_\zL$ on sections of $E$:
$$\zi([X,Y]_\zL)=\{\zi(X),\zi(Y)\}_\zL\,,$$
where $\zi(X)$ is the linear function on $E^*$ associated with $X\in\Sec(E)$. This bracket, in general, does not satisfy the Jacobi identity (it may even be non-skew-symmetric) and possesses two anchors, the left one and the right one.
This structure, called \emph{general algebroid}, was introduced and studied in \cite{Grabowski:1999}. If the tensor $\zL$ is skew-symmetric, the corresponding general algebroid is called \emph{skew algebroid}. A skew algebroid for which the anchor map is a morphism of algebroids: $\zr[X,Y]=[\zr(X),\zr(Y)]_{vf}$, where $[\cdot,\cdot]_{vf}$ is the bracket of vector fields, we call an \emph{almost Lie algebroid}. On an almost Lie algebroid $E$ one can develop the concept of homotopy of $E$-paths and Pontryagin Maximum Principle \cite{Grabowski:2011}.
\begin{example} Let $D$ be a vector subbundle of a Lie algebroid $E$ which equipped with the bracket $[\cdot,\cdot ]_\zL$. Suppose additionally that $E$ is equipped with a smooth symmetric tensor field $g\in\Sec(E^*\ot E^*)$ which induces a scalar product in the fibers (`Riemannian structure' on $E$), so that $E=D\oplus D^\perp$ and $E^*=D^*\oplus(D^\perp)^*$. If $p_D$ is the orthogonal projection $p_D:E\to D$, then we have an induced skew algebroid bracket $[\cdot,\cdot ]_D$ on $D$, which for $X,Y\in\Sec(D)$ reads
$$[X,Y]_D= p_D([X,Y]_\zL)\,.$$
If $p_{D^*}:E^*\to D^*$ is the orthogonal projection of $E^*$ onto $D^*$, then
the corresponding Leibniz tensor on $D^*$ is $(p_{D^*})_*(\zL)$.
The bracket $[X,Y]_D$ does not satisfy the Jacobi identity in general, and is used to formulate a nice geometric description of dynamics for non-holonomic constraint $D$ and mechanical Lagrangians \cite{Grabowski:2009a}.
\end{example}
\noindent
Recall that a two-form $\omega=\omega_{ij}(x)\rmd x^i\ot \rmd x^j$ is compatible with a graded bundle structure on $\zt:F\to M$ of degree $k$ if and only if $\zw$ has degree $k$. In this case, $\deg(\zw_{ij})=k-w_i-w_j$. A natural definition of weighted pseudo-Riemannian structures then is the following.
\begin{definition} A weighted pseudo-Riemannian structure on a graded bundle $F$ of degree $k$ is a pseudo-Riemannian structure $\zm$ on $F$ such that the symmetric two-form $\zm$ is of degree $k$.
\end{definition}

\begin{proposition} A $(0,2)$-tensor $\zw$ on a graded bundle $F$ is weighted if and only if
$$\zw^\flat:\sT F\to\sT^*[k]F\,,\quad \zw^\flat(f(x)\pa_{x^l})=f(x)\zw_{lj}(x)\xd x^j\,,$$
is a morphism of $\GrL$-bundles.
\end{proposition}
\begin{proof}
As $\zw^\flat(x^i,\dot x^j)=(x^i,\sum_l\dot x^l\zw_{lj}(x))$ and $\deg(\sum_l\dot x^l\zw_{lj}(x))=k-w_j=\deg(p_j)$,
the map $\zw^\flat$ preserves the degrees.

\end{proof}
\noindent Now, suppose that the map $\zw^\flat$ is an isomorphism of $\GrL$-bundles (e.g $\zw$ is a symplectic form). Then, $(\zw^\flat)^{-1}:\sT F\to\sT^*F$ is also an isomorphism of $\GrL$ bundles and corresponds to a Leibniz tensor field $\zL$, $\zL^\sharp=(\zw^\flat)^{-1}$.
According to Proposition \ref{p2}, $\zL$ is of degree $-k$, so weighted.
\begin{corollary}
If a weighted Poisson tensor $\zL$ is symplectic, then the corresponding symplectic form is also weighted.
\end{corollary}

\subsection{Weighted contact structures}
It is clear that a \emph{weighted contact form} $\za$ on a graded bundle of degree $k$ is a contact form which is homogeneous of degree $k$. Let us recall that a contact structure $C$ on a manifold $M$ is a co-rank $1$ distribution on $M$ which is `completely non-integrable'. Such distributions are locally kernels of local contact $1$-forms. This implies that the dimension of $M$ is odd (another approach to graded contact geometry and Jacobi structures one can find in \cite{Mehta:2013}).

It was shown in \cite{Grabowski:2013} that a $1$-form $\za$ on $M$ is contact if and only if the canonical symplectic form $\zw_M$ on $\sT^*M$, restricted to the line subbundle $L_\za\to M$ generated by the image of $\za$, is symplectic form $\zw_\za$ on $L^\ti_\za=L_\za\setminus \{0_M\}$, i.e. on $L_\za$ with the zero-section removed. Note that, $L_\za^\ti$ is canonically a $\R^\ti$-principal bundle over $M$, and the subbundle $L_\za\subset \sT^*M$ may be viewed as the annihilator of $C=\Ker(\za)$. The map $I_\za:\R^\ti\ti M\to L_\za^\ti$ given in local coordinates by
$I_\za(s,x)=s\za(x)$, is a diffeomorphism. For $\za=\za_i(x)\xd x^i$,
$$I_\za^*(\zw_\za)=\xd(t\za_i)\we\xd x^i=\za_i(x)\,\xd t\we \xd x^i+t\,\xd\za_i\we\xd x^i=\xd t\we\za+t\,\xd\za\,.$$
The $\R^\ti$-action on $L^\ti_\za$ transformed to $\R^\ti\ti M$ reads $m_t(s,x)=(st,x)$.
This is the reason why in \cite{Grabowski:2013} contact structures were identified as
\emph{symplectic principal $\R^\ti$-bundles}. The latter is canonically $\R^\ti$-principal bundle $P$ over $M$ and the symplectic form $\zw$ on $P$ is homogeneous of degree $1$ with respect to the $\R^\ti$-action: $m_t^*(\zw)=t\cdot\zw$.

It is easy to see that if $M$ is a graded bundle with the homogeneity structure $h$ of degree $k$, then $L_\za^\ti$ is a graded bundle with respect to the homogeneity structure $\hat h_t(s,x)=(s,h_t(x))$.
The $1$-form $\za$ is homogeneous of degree $\zl$ on $M$ if and only if $\zw_\za$ is homogeneous of degree $\zl$ on $L_\za^\ti$.

\begin{proposition}
Let $C\subset\sT F$ be a contact structure on a graded bundle $F$. If $\za$ and $\zb$ are local contact forms of degrees $w_\za$ and $w_\zb$, respectively, each of them generating locally $C^o$,  then $w_\za=w_\zb$.
\end{proposition}
\begin{proof}
There is a nowhere-vanishing local function $f$ such that $\zb=f\za$. We have then
$$w_\zb f\za=w_\zb\zb=\Ll_{\nabla_F}\zb=\nabla_F(f)\za+f\Ll_{\nabla_F}\za=(\nabla_F(f)+w_\za f)\za\,.$$
Hence, $\nabla_F(f)=(w_\zb-w_\za)f$, i.e. $f$ is of weight $w_\zb-w_\za$ and is a polynomial in homogeneous local coordinates. Since $\nabla_F$ is linear, it is clear that the constant term in this polynomial must be 0, so $f$ vanishes at 0. But then $f\za$ vanishes at 0 and therefore it cannot be a contact form.

\end{proof}
\noindent The above proposition justifies the following definition.
\begin{definition}
A contact structure $C\subset\sT F$ on a graded bundle $F$ of degree $k$ we call \emph{homogeneous of degree $r$} if, in a neighbourhood of each point $p\in F$, the line bundle $C^o$ is generated by a homogeneous contact form of degree $r$. We say that the contact structure $C$ is  \emph{weighted} if homogeneous local generators of $C^o$ are weighted contact forms.
\end{definition}
\noindent This definition immediately implies that a weighted contact structure is a weighted distribution on $F$ and the annihilator $C^o$ is a graded subbundle in $\sT^*F$. We have also
\begin{corollary}
A contact structure $C\subset\sT M$ is weighted if and only if the symplectic form $\zw$ on $(C^o)^\ti$ is weighted.
\end{corollary}

\subsection{Weighted Poisson-Nijenhuis structures}
Let $F$ be a graded bundle and	
$$ {N}=\displaystyle  N_j^{i}(x) \frac{\pa}{\pa x^i}\ot \rmd x^j$$ be a weighted Nijenhuis tensor on $F$.
This means that the degree of ${N}$ is zero, i.e. $\deg(N_j^{i})=w_i-w_j$. It is easily seen that this is equivalent to the fact that the associated map
$\tilde{{N}}:\sT F\to\sT F$ defined by
$$(x^i,\dot x^j)\circ \tilde{{N}}=(x^i, N^{j}_{l}(x)\dot x^l)$$
is a morphism of $\GrL$-bundles. Indeed,
$\deg( N^{j}_{l}\,\dot x^l)=w_j-w_l+w_l=w_j=\deg(x^j)$.
		
Let us recall now that a Poisson-Nijenhuis manifold is a manifold $F$ equipped with a Poisson tensor $\zL=\zL^{ij}(x)\pa_{x^i}\ot\pa_{x^j}$, $\zL^{ij}=-\zL^{ji}$, and a Nijenhuis tensor $ N$ that are compatible, which means that
	\be\label{PN}\tilde{  N} \circ \zL^{\sharp} = \zL^{\sharp} \circ\tilde{  N}^t\,,\ee
	and $$
C(\zL, N)(\alpha,\beta)=\left[ \alpha,\beta\right]_{  N\zL}- \left[ \alpha,\beta\right]^{ N^{t}}_\zL=0,\quad \forall \alpha, \beta \in \Omega^1(F)\,.$$
Here, $\tilde{  N}^{t}:\sT^{*}F \to \sT^{*}F$ is the dual to $\tilde{  N}$ and
\[
	\left[ \alpha,\beta\right]^{{ N}^{t}}_\zL
	=[\tilde{  N}^{t}\alpha,\beta]_{\zL}+[\alpha,\tilde{  N}^{t}\beta]_{\zL} -\tilde{  N}^{t}\left[ \alpha,\beta \right]_{\zL}\,,	
	\]
where the bracket $[\cdot,\cdot]_{\zL}$ is the bracket of 1-forms defined by the Poisson bivector $\zL$.
	Similarly, the bracket $\left[ \alpha,\beta\right]_{{ N}\zL}$ is the bracket of 1-forms defined by the Leibniz tensor
$${ N} \zL=\zL^{il}(x)N^j_l(x)\pa_{x^i}\ot\pa_{x^j}$$
which is the Leibniz tensor generating the linear map $\tilde{  N} \circ \zL^{\sharp}:\sT F\to\sT F$;
(for more details see, \cite{Kosmann:1990,Magri:1984}). Condition (\ref{PN}) means that the tensor $ N\zL$
is skew-symmetric.

\noindent It is known that $C(\zL,  N)$ is a  $(2,1)$-tensor field, called the \emph{concomitant} of $\zL$ and $ N$.
\begin{definition}
A \emph{weighted Poisson-Nijenhuis structure} on a graded  bundle $F$ of degree $k$ is a Poisson-Nijenhuis structure $(\zL, N)$ whose Poisson and Nijenhuis structures are weighted, i.e. $\zL$ is of degree $k$ and $ N$ is of degree $0$.
\end{definition}
\begin{theorem}
If a Leibniz tensor $\zL$ and a $(1,1)$-tensor $ N$ on $F$ are weighted, then  $ N\zL$ and $C(\zL,  N)$ are also weighted tensors.
\end{theorem}
\begin{proof}
We calculate easily the degree of $ N\zL$ taking into account that $\deg(\zL^{il})=w_i+w_l-k$ and $\deg(N^j_l)=w_j-w_l$:
$$\deg(\zL^{il}N^j_l\pa_{x^i}\ot\pa_{x^j})=w_i+w_l-k+w_j-w_l-w_i-w_j=-k\,.$$
We have
		$$(x^i,\dot x^j)\circ (\tilde N \circ \zL^{\sharp} )=(x^i,(\zL^{js} N^l_{s})p_l)\,,$$
so the Leibniz tensor
$$ N \zL=\displaystyle(\zL^{js} N^l_{s})\frac{\pa}{\partial x^j}\ot \frac{\pa}{\partial x^l}$$
is of the degree $-k$ and thus is a weighted tensor.
		
For a $(2,1)$-tensor field
$$C(\zL,  N)=\displaystyle C^{ij}_s(x)\frac{\pa}{\pa x^i} \ot \frac{\pa}{\pa x^j}\ot dx^s$$ to be of degree $-k$, it is required that $\deg(C^{ij}_s)=w_i+w_j-w_s-k$. In the coordinate expression we have \cite{Kosmann:1990}
		\[
	C^{ij}_s=\Lambda^{lj}\pa_{x^l}{ N}^i_s+\Lambda^{il}\pa_{x^l}{ N}^{j}_s(x)-{ N}^l_s\pa_{x^l}\Lambda^{ij}
+{ N}^j_l\pa_{x^s}\Lambda^{il}-\Lambda^{lj}\pa_{x^s}{ N}^i_l.	
		\]
Using the fact that $\deg(\pa_{x^i}(f))=\deg(f)-w_i$, by direct calculations we get $\deg(C^{ij}_s)=w_i+w_j-w_s-k$.	

\end{proof}
\noindent Any  $(2,1)$-tensor field of degree $-k$,
\be\label{e7}C=C^{ij}_s(x)\frac{\pa}{\pa x^i} \ot \frac{\pa}{\pa x^j}\ot \xd x^s\ee 
defines by contraction a vector bundle morphism
$$\tilde C:\we^2\sT^*[k]F\to\sT^*[k]F$$	
and \emph{vice versa}.
\begin{proposition}	
A $(2,1)$-tensor field (\ref{e7}) on a graded bundle $F$ of degree $k$ is weighted if and only if the associated map $\tilde C$ is a morphism of $\GrL$-bundles.
\end{proposition}
\begin{proof}
The tensor $C$ is weighted if and only if $\deg(C^{ij}_s)=w_i+w_j-w_s-k$.
As the map $\tilde C$ in local coordinates looks like
$$(x^i,p_j)\circ\tilde C=(x^i,\sum_{l,s}C^{ls}_j(x)(p_{ls}))\,,$$
where $p_{ls}$ are linear coordinates in $\we^2\sT^*[k]F$,
$$\deg\left(\sum_{l,s}C^{ls}_j(p_{ls})\right)=w_l+w_s-w_j-k+(k-w_l)+(k-w_s)=k-w_j=\deg(p_j)\,.$$

\end{proof}

\subsection{Weighted algebroids and groupoids}
Motivated by our papers \cite{Bruce:2015,Bruce:2016}, we propose the following.
\begin{definition}
A \emph{weighted groupoid of degree $k$} (\emph{weighted algebroid of degree $k$}) is a graded bundle $(F,h)$ of degree $k$ over a submanifold $M$ equipped additionally with a Lie groupoid structure $F\rightrightarrows B$ (resp. Lie algebroid structure over $B$) such that the maps $h_t$ act as Lie groupoid (resp. Lie algebroid) morphisms for all $t\in\R$.
\emph{Morphisms} of weighted groupoids (algebroids) are morphisms $\zf:F_1\to F_2$ of Lie groupoid structures (resp. Lie algebroid structures) which intertwine the corresponding homogeneity structures, $h^2_t\circ\zf=\zf\circ h^1_t$.
\end{definition}
\begin{proposition}
The base $B$ of a weighted groupoid (Lie algebroid) $\zt:F\to M$ is canonically a graded subbundle in $F$. Similarly, $M$ is canonically a Lie subgroupoid (Lie subalgebroid) of $F$. We have a commutative diagram for the weighted groupoid $F$

\begin{equation*}
\xymatrix @R20mm @C20mm{
   F \ar[r]^{h_0} \ar@<-2pt>[d] \ar@<2pt>[d] & M \ar@<-2pt>[d] \ar@<2pt>[d]\\
   B \ar[r]^{h_0} & M\cap B\,.
}
\end{equation*}

\end{proposition}

\medskip\noindent
\begin{proof} The proof is for weighted groupoids. For weighted Lie algebroid it is analogous. For weighted groupoids, $h_t$ is a Lie groupoid morphism for all $t\in\R$, that is
\[
h_t(g)=h_t(\za(g)\cdot g)=h_t(\za(g))\cdot h_t(g)\,,
\]
where $\za:F\to B$ is the source map, which implies $h_t(\za(g))=\za (h_t(g))$. Similarly, $h_t(\zb(g))=\zb (h_t(g))$ for the target map $\zb:F\to B$. The base $B$ is therefore invariant with respect to all $h_t$, so it is a graded subbundle of $F$, $B\to h_0(B)$. As $h_0\circ\za=\za\circ h_0$ and $h_0\circ\zb=\za\circ h_0$, we have $\za(M)=\zb(M)=h_0(B)$. Since $\za$ is a surjective submersion of $F$ onto $B$ and $h_0$ is a surjective submersion of $B$ onto its submanifold $h_0(B)$, the map $h_0\circ\zb:F\to h_0(B)$ is a surjective submersion. This means that $\sT h_0\circ\sT\za$ has $\sT h_0(B)$ as its image. But $h_0\circ\za=\za\circ h_0$, so $\za\circ h_0$ is also a surjective submersion. Hence,
$$(\sT\za\circ\sT h_0)(\sT F)=\sT\za(\sT M))=\sT(\za(M))\,,$$
so $\za_{|M}:M\to\za(M)=h_0(B)$ is a surjective submersion. Similarly, $\zb_{|M}$ is a surjective submersion. It remains to show that the groupoid multiplication $g_1\cdot g_2\in M$ if only $g_1,g_2\in M$ and that $\za(M)=\zb(M)=h_0(B)$ equals $M\cap B$. For, suppose $g_1,g_2\in M$. Because $h_0$ is a Lie groupoid morphism, $$h_0(g_1\cdot g_2)=h_0(g_1)\cdot h_0(g_2)=g_1\cdot g_2\,,$$ that shows $g_1\cdot g_2\in M$.
Finally, it is obvious that $\za(M)=h_0(B)=(h_0\circ\za)(F)\subset M\cap B$. Let us take $p\in M\cap B$. Then we have
$$(h_0\circ\za)(p)=h_0(\za(p))=h_0(p)=p\,,$$
which means that $(M\cap B)\subset h_0(B)$.

\end{proof}

\noindent If the Lie theory is concerned, we have the following.
\begin{theorem}\cite[Theorem 4.1]{Bruce:2015}
The infinitesimal part of a weighted groupoid $G$ of degree $k$ with respect to a homogeneity structure $h$ on $G$ is the Lie algebroid $\Lie(G)$ which is weighted of degree $k$ with respect to the induced homogeneity structure $\Lie(h_t):A(G)\to A(G)$, where $\Lie(h_t)$ is the Lie algebroid morphism associated with the Lie groupoid morphism $h_t:G\to G$ .
\end{theorem}
\begin{remark}
Note that in \cite{Bruce:2016} the degree of a weighted algebroid is one degree smaller than here. Of course, this definition  requires \emph{implicite} that $F$ is equipped additionally with a vector bundle structure $F\to N$ associated with a homogeneity structure $h'$, which makes $F$ into a $\GrL$-bundle. Moreover, as $h_t\circ h'_s=h'_s\circ h_t$,  all $h_t$ map $N$ into {\cbl $N$,} and $N$ is a graded bundle over $M'=h_0(N)=h'_0(M)$. Similarly, $h'_t$ maps $M$ into $M$, and $M$ is canonically a vector bundle over $M'$. Note that, for $t\ne 0$, the Lie algebroid morphism $h_t:F\to F$, is a morphism of Lie algebroids over the diffeomorphism $h_t:N\to N$, so it maps sections onto sections of the vector bundle $h'_0:F\to N$  and is characterized by
$$h_t[\zs_1,\zs_2]=[h_t(\zs_1),h_t(\zs_2)]$$ for sections $\zs_1,\zs_2$, where
$$h_t(\zs)(x)=h_t(\zs(h_{t^{-1}}(x)))\,.$$ Here, we understand $h_t$, with some abuse of notation, as a diffeomorphism of $F$ as well as a diffeomorphism of $N$.

\smallskip\noindent
Indeed, the property required for the anchor $\zr:F\to\sT N$, namely
\be\label{e5}\sT h_t\circ\zr=\zr\circ h_t\,,\ee
follows automatically. We have
$$h_t\left(f[\zs_1,\zs_2]+(\zr(\zs_1)(f))\circ h_{t^{-1}}\circ h_t(\zs_2)\right)=h_t[\zs_1,f\zs_2]=[h_t(\zs_1),h_t(f\zs_2)]=[h_t(\zs_1),f\circ h_{t^{-1}}\circ h_t(\zs_2)]$$
which implies
$$\zr(\zs_1)(f)\circ h_{t^{-1}}=\zr(h_t(\zs_1))(f\circ h_{t^{-1}})$$
for all $f\in C^\infty(N)$, and it is equivalent to (\ref{e5}).

\end{remark}
\begin{example}(\cite[Proposition 4.12]{Bruce:2016}  and \cite[Example 3.10]{Bruce:2015})
If $F$ is a Lie groupoid (Lie algebroid), then $\sT^kF$ is canonically a weighted groupoid (algebroid) of degree $k$.
The Lie algebroids $\sT^kF$ are examples of higher Lie algebroids in the sense of J\'o\'zwikowski and Rotkiewicz \cite{Jozwikowski:2015}.
\end{example}
\noindent For the Lie groupoid structure on $\sT^kG$ we refer to \cite[12.13]{Kolar:1996}. For a Lie algebroid structure see e.g. \cite[Theorem 3]{Wamba:2012}.
\begin{proposition}(\cite[Proposition 2.19 and 3.6]{Bruce:2015}
Let $F_k\to M$ be a weighted groupoid (algebroid over $N$) of degree $k$. Then the reduced graded bundles  $F_i$ are canonically  weighted groupoids (algebroids) of degree $i$, $i=0,\dots,k$, and the tower of affine fibrations (see (\ref{afffib}))
$$
F=F_{k} \stackrel{\tau^{k}}{\longrightarrow} F_{k-1} \stackrel{\tau^{k-1}}{\longrightarrow}   \cdots \stackrel{\tau^{3}}{\longrightarrow} F_{2} \stackrel{\tau^{2}}{\longrightarrow}F_{1} \stackrel{\tau^{1}}{\longrightarrow} F_{0}=M
$$
consists of Lie groupoid (Lie algebroid) morphisms. In particular, $M=h_0(F_k)$ is a Lie subgroupoid (Lie subalgebroid) in $F$ and $h_0:F_k\to M$ is a Lie groupoid (Lie algebroid) morphism.
\end{proposition}
\begin{remark}
The bundles $F_{i}\to F_{i-1}$ are affine bundles, but for a Lie algebroid $F$, each $F_i$ has also a vector bundle structure over  $N_i$ and maps $\zt^i:F_{i}\to F_{i-1}$ are vector bundle morphisms.
\end{remark}
The definition of a weighted algebroid can be extended to \emph{weighted general algebroid} in an obvious way.
A slight modification of \cite[Proposition 4.4]{Bruce:2016} gives the following characterizations of weighted general algebroids of degree $k$.
\begin{proposition}
Let $F$ be a $\GrL$-bundle of degree $k$ with the graded bundle projection $F\to M$ and vector bundle projection $F\to N$. Let $F^*$ be the dual of $F$ with respect to the vector bundle structure. There is a one-to-one correspondence between weighted general algebroid structures on $F$ and
\begin{enumerate}
\item morphisms of triple graded bundles $\ze:\sT^*[k]F\to\sT F^*[k]$, covering the identity on the $\GrL$-bundle $F^*[k]$.
\item $2$-contravariant tensors $\zL$ on $F^*[k]$ of bi-degree $(-k,-1)$.
\item a general algebroid bracket $[\cdot,\cdot]_\zL$ on sections of $F\to N$ which is of degree $-k$, i.e. the bracket $[\zs_1,\zs_2]_\zL$ of sections $\zs_1$ and $\zs_2$ of degrees $w_1$ and $w_2$, respectively,
    is of degree $w_1+w_2-k$.
\end{enumerate}
\end{proposition}

By our definition of $\VB$-structures, weighted groupoids (algebroids) of degree 1 are called $\VB$-groupoids (algebroids).
The original concept of $\VB$-algebroid was introduced by Pradines \cite{Pradines:1974,Pradines:1988} and it has been
further studied by Mackenzie \cite{Mackenzie:2005}, Gracia-Saz and Mehta \cite{Gracia:2010} among others.
The concept of a $\VB$-groupoid one can find already in \cite{Mackenzie:1992,Mackenzie:2000} and
\cite[Section 2.1]{Mackenzie:2005}, where they are understood as double Lie groupoids for which one structure is a vector bundle. The $\VB$-algebroids and $\VB$-groupoids have shown to be especially important in the
infinitesimal description of Lie groupoids equipped with multiplicative geometric
structures and as geometric models for representations up to homotopy \cite{Bruce:2018,Gracia-Saz:2010,Gracia-Saz:2017}.
The original definitions are quite complicated and refer to $\VB$-groupoids ($\VB$-algebroids) as  Lie groupoid (Lie algebroid) objects in the category of vector bundles.
Only in \cite{Bursztyn:2016} it was discovered that the use of vector bundle characterization in terms of regular homogeneity structures, i.e. of degree $1$ \cite{Grabowski:2009}, substantially simplifies the definition. As a result, we have an equivalence of traditional definitions with the ones proposed in this paper.

\subsection{Weighted principal bundles}
Motivated by our paper \cite{Bruce:2016} we propose the following.
\begin{definition} A \emph{weighted $G$-principal bundle} $\tau:P\to M$ is a $G$-principal bundle equipped additionally with a homogeneity structure $h$ such that the $G$-action  and $\R$-action commute.
\end{definition}

\begin{proposition}\label{prop:principal}
If the principal $G$-action and homogeneity structure on a $G$-principal bundle $\tau:P\to M$ commute, then $P_0=h_0(P)$ is $G$-invariant, and therefore is itself a principal bundle over $M_0=P_0/G$, and $h_0:P\to P_0$ is a principal bundle morphism. Moreover $P=M\times_{M_0}P_0$ is the pull-back bundle $(h^M_0)^*P_0$ and the $G$-action on $M\times_{M_0}P_0$ reduces to the action on the factor $P_0$.
\end{proposition}
\begin{proof}
Let $p_0$ be an element of $P_0$, then for any $g\in G$ we have $h_0(p_0g)=h_0(p_0)g=p_0g$ which means that $p_0g\in P_0$. The submanifold $P_0$ is then invariant with respect to the $G$-action and it follows that $P_0$ is composed of fibers of the principal bundle $P$. Let $x$ be an element of $M$. For $p,p'\in\tau^{-1}(x)$ we have $p'=pg$ for some $g\in G$ and therefore $$\tau(h_t(p'))=\tau(h_t(pg))=\tau(h_t(p)g)=\tau(h_t(p)).$$
The $\R$-action descends then to the homogeneity structure $h^{M}$ on $M$ making it a graded bundle $M\rightarrow M_0$ (see Theorem \ref{t5}). The principal bundle $P_0$ has $M_0$ as base, $\zt_0:P_0\to M_0$.
The fact that $G$-action and $\R$-action commute assures that $h_0$ is a $G$-principal bundle morphism with $h^M_0: M\rightarrow M_0$ as a base map. In the diagram
$$\xymatrix{
 P\ar[r]^{h_0}\ar[d]^{\tau} & P_0\ar[d]^{\tau_0} \\
M\ar[r]^{h^M_0} & M_0
}$$
the horizontal arrows represent a principal bundle morphism and the vertical arrows represent a graded bundle morphism. Since for given $x\in M$ and $p_0\in P_0 $, such that
$h_0^M(x)=\tau_0(p_0)$, there is only one $p\in P$ for which $\tau(p)=x$ and $h_0(p)=p_0$,
$$P=M\times_{M_0}P_0\,.$$
Indeed, since principal bundles are locally trivial, we can work with trivial bundles. We have $P=M\ti G$ and $P_0=M_0\ti G$, so $M\ti P_0=M\ti(M_0\ti G)$ and
$$M\ti_{M_0}P_0=M\ti_{M_0}(M_0\ti G)=M\ti G=P\,.$$

\end{proof}

\begin{remark}
A trivial principal bundle $P=M\ti G$ is weighted if and only if $h_0^M:M\to M_0$ is a graded bundle with the homogeneous structure $h^M$, and the $h_t$ action and $G$-action on $P$ read
$$h_t(x,g)=(h_t^M(x),g)\quad\text{and}\quad (x,g)\cdot g'=(x,gg')\,\quad (x,g)\in M\ti G\,,$$
i.e. $P=M\ti G$ is a trivial $G$-principal bundle over a graded bundle $M$.
\end{remark}

\begin{example}\label{44}
Let $P\rightarrow M$ be a principal bundle with the structure group $G$. The group $G$ acts on $\sT P$ by the lifted action. More precisely, if $\phi_g$ denotes the map
$P\ni p\mapsto pg\in P$ then the lifted action of $G$ on $\sT P$ is composed of maps $\sT\phi_g$. It is well known that $A(P)=\sT P\slash G$ is a vector bundle over $M$ with the structure of a Lie algebroid, called the \emph{Atiyah algebroid}. The bundle $\sT P\rightarrow A(P)$ is a weighted principal bundle of degree one, with respect to the canonical homogeneity structure of the tangent bundle $\sT P$ over $P$. Indeed, the map $\sT\phi_g$ is a tangent map, so it is linear on fibers over $P$; in particular, it commutes with multiplication by reals. In the diagram
$$\xymatrix{
\sT P\ar[r]\ar[d] & P\ar[d] \\
A(P)\ar[r] &  M
}$$
the horizontal arrows denote graded bundles, while the vertical arrows denote principal bundles. In the local trivialization $P\simeq M\times G$ the $\phi_g$ action reads
$$\phi_g(x,h)=(x, hg).$$
Applying the tangent functor to $P\simeq M\times G$, we get $\sT P\simeq \sT M\times G\times \mathfrak{g}$, where the element $(v,h,X)$ is tangent to the curve
$t\mapsto (\gamma(t), h\exp(tX))$ with $\dot\gamma(0)=v$. The local formula for $\sT\phi_g$-action reads
$$\sT\phi_g(v,h,X)=(v, hg, Ad_{g^{-1}}(X)),$$
and $A(P)$ in the local trivialization is isomorphic to $\sT M\times_M \mathrm{ad} P$.
Here, $\mathrm{ad} P$ is the adjoint bundle of $P$ associated with the Lie algebra $\mathfrak g$ of $G$ and equipped with the adjoint action of $G$, $\mathrm{ad} P=(P\ti{\mathfrak g})/G$.
The fundamental vector fields of the group action on $P$ are given by the formula
$$\xi_Y(p)=\frac{d}{dt}_{|t=0}p\cdot\exp(tY),$$
which in a local trivialization reads
$$\xi_Y(x,h)=(0_x,h,Y),$$
where $0_x$ is a zero-vector at $x\in M$. The corresponding fundamental vector field of the lifted action is the complete lift $\xi_Y^{(c)}$. We can express it again in local trivialization, applying once more the tangent functor to $\sT P\simeq \sT M\times G\times \mathfrak{g}$ with the result $\sT\sT P\simeq \sT\sT M\times G\times\mathfrak{g}\times\mathfrak{g}\times\mathfrak{g}$. The complete lift $\xi_Y^{(c)}$ in this trivialization reads
$$\xi_Y^{(c)}(v,h,X)=(0_v,h,X;\,Y,[X,Y]),$$
where again $0_v$ is the zero-vector attached at $v\in\sT M$. It commutes with the Euler vector field $\nabla_{\sT P}$
which in this trivialization reads
$$\nabla_{\sT P}(v,h,X)=(\nabla_{\sT M}(v),h,X;\,0, X).$$
\end{example}

\begin{example}
A similar example, this time of a weighted principal bundle of degree k, we get applying the $\sT^k$-functor to the map $\phi_g$ for every $g\in G$. Dividing $\sT^kP$ by the group action, we get $A^k(P)$, i.e. the $k$-th prolongation of the Atiyah algebroid. In the diagram
$$\xymatrix{
\sT^k P\ar[r]\ar[d] & P\ar[d] \\
A^k(P)\ar[r] & M\,,
}$$
the horizontal arrows denote graded bundles while the vertical arrows denote principal bundles. Let us again employ a local trivialization $P\simeq M\times G$. For clarity of notation, we will discuss in the trivialization the case $k=2$. Using the left trivialization for $\sT^2 G$ we get
$$\sT^2 P\simeq \sT^2 M\times G\times\mathfrak{g}[1]\times\mathfrak{g}[2]\,,$$
where we have indicated the weights of the Lie algebra components. The lifted  $G$-action on $\sT^2 P$ in this trivialization reads
$$\sT^2\phi_g(v,h,X,Z)=(v,hg,Ad_{g^{-1}}X,Ad_{g^{-1}}Z)\,,$$
and $A^2(P)$ is then isomorphic to $\sT^2M\times_M\mathrm{ad}P[1]\times_M\mathrm{ad}P[2]$, where again we have indicated the weights of the components.
Since the tangent bundle $\sT\sT^2 P$ can be written as
$$\sT\sT^2 P\simeq \sT\sT^2 M\times G\times\mathfrak{g}[1]\times\mathfrak{g}[2]\times \mathfrak{g}[0,1]\times\mathfrak{g}[1,1]\times\mathfrak{g} [2,1]\,,$$
the fundamental vector field of $G$-action on $\sT^2 P$ reads
$$\xi_Y^{(2)}(u,h,X,Z)=\left(0_u, h,X,Z;\, Y, [X,Y],[Z,Y]\right)\,,$$
while the Euler vector field at point $(u,h,X,Z)$ reads
$$\nabla_{\sT^2 P}(u,h,X,Z)=(\nabla_{\sT^2 M}(u),h,X,Z;\,0, X, 2Z).$$
The other lifts of fundamental vector fields  are
$$\xi_Y^{(1)}(u,h,X,Z)=(0_u,h,X,Z,0,Y,[X,Y])$$
and
$$\xi_Y^{(0)}(u,h,X,Z)=(0_u,h,X,Z,0,0,2Y)\,.$$
\end{example}
\begin{example} Let $P\rightarrow M$ be a weighted principal bundle of degree $k$ with the structure group $G$. We have discussed the principal bundle structure on $\sT P$ over the Atiyah algebroid $A(P)$ with structure group $G$. $\sT P$ carries also a principal bundle structure over $\sT M$, this time with structure group $\sT G$.
Let $h$ denotes the homogeneity structure of degree $k$ on $P$. The fact that $P$ is a weighted principal bundle means that $h_t(pg)=h_t(p)g$. Let now $\gamma$ be a curve on $P$ and $\eta$ a curve on  $G$. For every value of the real parameter $s$ we have $h_t(\gamma(s)\eta(s))=h_t(\gamma(s))\eta(s)$, therefore for $v$ being the vector tangent to $\gamma$ at $s=0$, and $u$ being the vector tangent to $\eta$ at $s=0$, we have
$$\sT h_t(v\cdot u)=\sT h_t(v)\cdot u\,,$$
where $\cdot$ denotes the action of the tangent group $\sT G$ on $\sT P$. This shows that $\sT P$ is a weighted principal bundle of degree $k$ with respect to the lifted homogeneity structure $\dt h$. According to Proposition \ref{prop:principal}, $P$ is diffeomorphic to the fibered product $P\simeq M\times_{M_0} P_0$, where $P_0=h_0(P)$ is a principal bundle over $M_0=h^M_0(M)$. Applying the tangent functor, we get
$$\sT P\simeq \sT M\times_{\sT M_0}\sT P_0, $$
where $\sT P_0\rightarrow \sT M_0$ is a principal bundle with structure group $\sT G$, and $\sT M$ is a graded bundle of degree $k$ with respect to the lifted homogeneity structure $\dt h^M$.
\end{example}

\subsection{Weighted principal connections}
\begin{definition} A \emph{weighted principal connection} in a weighted principal bundle $\tau:P\to M$ is a principal connection such that the horizontal distribution is a weighted distribution.
\end{definition}

\begin{proposition}\label{48}
Weighted principal connections on a weighted principal bundle $P\rightarrow M$ are in a one-to-one correspondence with principal connections on the principal bundle $P_0\rightarrow M_0$, where $P_0=h_0(P)$ and $M_0=h_0^M(M)$ in the notation of Proposition \ref{prop:principal}. The connection one-form $\omega$ on $P$ and the curvature two-form $\Omega$ on $M$ are homogeneous forms of weight $0$. Moreover,
$$\omega=h_0^\ast\omega_0,\qquad \Omega=(h_0^M)^\ast\Omega_0\,,$$
for appropriate connection and curvature forms $\omega_0$ and $\Omega_0$ on the principal bundle $P_0$. In particular, the connection on $P$ is the pull-back of the connection on $P_0$ (cf. Proposition \ref{prop:principal}).
\end{proposition}
\begin{proof} From Proposition (\ref{prop:principal}) we know that weighted $G$-principal bundle is of the form $M\times_{M_0}P_0$, where $P_0$ is a $G$-principal bundle over $M_0$, and $M$ is a graded bundle over $M_0$ with the homogeneity structure $h^M$. The tangent bundle $\sT P$ is then isomorphic to $\sT M\times_{\sT M_0}\sT P_0$, more precisely
$$\sT P\simeq \sT M\times_{\sT M_0}\sT P_0=\{(v,u):\; v\in \sT M, u\;\in\sT P_0,\; \sT h^M_0(v)=\sT\tau_0(u)\}\,.$$
Let $\sH^0$ denote the horizontal distribution of a principal connection on $P_0$, i.e.
$$\sT_pP_0=\sV^0(p)\oplus\sH^0(p)\,,$$
where $\sV^0(p)$ is the subspace of vectors tangent at $p$ to the fibre of $P_0$, and $\sH^0(p)$ satisfies the condition $\sH^0(pg)=\sH^0(p)g$ for every $g\in G$. Let $\sH$ denote the following distribution on $P$,
$$\sH=(\sT h_0)^{-1}(\sH^0)=\{(v,u): v\in \sT M, u\;\in\sH^0,\; \sT h^M_0(v)=\sT\tau_0(u)\}.$$
We claim that $\sH$ is a weighted distribution and defines a principal connection in $P$. It is easy to check that $\sH$ is a distribution. Then, we observe that
$\sT h_t(v,u)=(\sT h_t^M(v), u)$, and since $\sT h^M_0(\sT h_t^M(v))=\sT h_0^M(v)$, it follows that $\sT h_t(v,u)$ is an element of $\sH$, for elements $(v,u)$ of $\sH$. It means that $\sH$ is weighted. For $g\in G$ we have
$(v,u)g=(v, ug)$, which gives us the $G$-invariance of $\sH$ provided $\sH^0$ is $G$ invariant. At each point $(x,p)$, the vectors tangent to the fibre over $x$ are of the form $\sV(x,p)=\{(0_x,u): u\in \sV^0(p)\}$. Therefore,
there is the splitting
$$\sT_{(x,p)} P\simeq \sV(x,p)\oplus\sH(x,p)\,,$$
defining the principal connection in $P$.

Conversely, assume that we have a principal weighted connection in $P$ with horizontal distribution $\sH$, i.e.
$$\sT_{(x,p)}P=\sV(x,p)\oplus\sH(x,p)\,.$$
Since the fiber of $P$ over $x$ equals the fiber of $P_0$ over $x_0=h^M_0(x)$, we have as previously $\sV(x,p)=\{(0_x,u): u\in \sV^0(p)\}$. We can then identify $\sV(x,p)$ with $\sV^0(p)$. We define $\sH^0(p)=\sH(x_0,p)\cap\sT_pP_0$. It is clear that $\sH^0(pg)=\sH^0(p)$, since $\sH$ is $G$-invariant. From the fact that the intersection of $\sV(x_0,p)$ and $\sH(x_0,p)$ is trivial, it follows that the intersection of $\sV^0(p)$ and $\sH^0(p)$ is also trivial. The principal weighted connection in $P$ defines then the principal connection on $P_0$. Moreover, $\sH$ is invariant with respect to $h_t$, which means that if $(v,u)\in \sH$, then $h_t(v,u)=(h^M_t(v),u)\in \sH$. It follows that $\sT h_0(\sH(x,p))=\sH^0(p)$ for all
$x$ over $x_0=\tau_0(p)$. The dimensional considerations show that $\sH=(\sT h_0)^{-1}(\sH^0)$.

We have shown that principal weighted connections on $P$ define principal connection on $P_0$, and the other way round. The horizontal distributions of these two connections satisfy $\sH=(\sT h_0)^{-1}(\sH^0)$, which means that the connection and curvature forms on $P$ are given by pull-backs of the connection and curvature forms on $P_0$. The latter contains weight-zero coordinates only, therefore $\omega$ and $\Omega$ are homogeneous of weight zero.

\end{proof}
\begin{example}
Let us consider a principal bundle $P$ of orthonormal oriented frames on a sphere $S^2\subset \R^3$, with the standard action of the group $SO(2)$. The base manifold of $P$ is of course $S^2$ itself. The projection will be denoted by $\pi: P\rightarrow S^2$. On the other hand, we can consider a point $n$ of the sphere as a unit vector $\vec{n}$ perpendicular to the sphere at the point $n$. This vector, together with an orthonormal frame at point $n$, form an orthonormal frame in $\R^3$. In this sense, $P$ is the space of orthonormal oriented frames in $\R^3$, and therefore the homogeneous space for the $SO(3)$ action. $\sT P$ can be now written as $P\times \mathfrak{so}(3)$, or even $P\times \R^3$ if we use the fact that the Lie algebra $\mathfrak{so}(3)$ is isomorphic to $\R^3$ with vector product $\vec{v}\times \vec{w}$ as the Lie bracket. The vector
$\vec{v}=v^1\vec{e}_1+v^2\vec{e}_2+v^3\vec{e}_3$ in the canonical basis in $\R^3$ corresponds to the matrix
$$v=\left[\begin{array}{ccc}
0 & -v^3 & v^2 \\
v^3 & 0 & -v^1 \\
-v^2 & v^1 & 0
\end{array}
\right]\in\mathfrak{so}(3).$$
The canonical scalar product in $\R^3$ in matrix form reads $(\vec{v}|\vec{w})=-\frac{1}{2}tr(vw)$.
Let $\vec{v}$ \ be an element of $\R^3$ being the tangent vector at $p\in P$ over the point $n\in S^2$. The map $\sT\pi: \sT P\rightarrow \sT S^2$ reads
$$\sT\pi(p,\vec{v})=(n,\vec{n}\times \vec{v}),$$
where we consider $\sT S^2$ as a subset of $\sT \R^3=\R^3\times \R^3$. The principal connection on $P$ can be defined by means of the canonical scalar product of $\R^3$:
the horizontal space at point $p$ over $n$ is $\sH_p=\langle \,\vec{n}\,\rangle^\perp$. One can check that this is indeed a principal connection on $P$. Due to the Cartesian product structure in $\sT P$, we have a distinguished set of vector fields on $P$, namely constant vector fields: $X_{\vec{v}}(p)=(p,\vec{v})$. One can check that the Lie bracket of such vector fields $X_{\vec{v}}$ and $X_{\vec{w}}$ is also a constant vector field  $X_{\vec{v}\times \vec{w}}$. The horizontal part of $X_{\vec{v}}$ reads
$$p\longmapsto (p,\,\vec{v}-(\vec{n}|\vec{v})\vec{n})\,,$$
while the connection one-form $\omega$ and curvature two-form $\Omega$ are given by
$$\omega(p,\vec{v})=(\vec{n}|\vec{v})\vec{n}, \qquad \Omega((p,\vec{v}),(p,\vec{w}))=-(n|\vec{v}\times \vec{w})\vec{n}.$$
In the above formula $\pi(p)=n$.
The values of $\omega$ and $\Omega$ are vertical vectors that can be identified with $\mathfrak{so}(2)\simeq \R$.

Let us now follow the example (\ref{44}) and consider $\sT P$ as an $SO(2)$-principal bundle over the Atiyah algebroid of the bundle $P$, i.e. $A(P)\simeq S^2\times \R^3$. In the diagram
$$\xymatrix{
P\times \R^3\ar[r]\ar[d] & P\ar[d] \\
S^2\times \R^3\ar[r] & S^2
}$$
the horizontal arrows denote graded bundles of weight $1$, i.e. vector bundles, while the vertical arrows denote principal $SO(2)$-bundles. According to Proposition \ref{48}, the horizontal distribution on $\sT P$ is the inverse image of the horizontal distribution on $P$ by $h_0$, which in this case coincides with
$\tau_P$. Since $\sT P\simeq P\times \R^3$ is a trivial bundle, we have
$$\sT \sT P\simeq P\times \R^3\times \R^3\times\R^3\,.$$
In this trivialization, $\tau_{\sT P}$ is the projection onto the first and second factor, and $\sT\tau_P$ is the projection onto the first and third factor. The horizontal distribution
$\sH^{\sT}\subset\sT\sT P$ of the principal connection on the bundle $\sT P\rightarrow A(P)$ reads
$$\sH^{\sT}=(\sT\tau_P)^{-1}(\sH)=\left\{(p,\vec{u},\vec{v},\vec{w}):\; (\vec{v}|\vec{n})=0, \,\vec{n}=\pi(p) \right\}.$$
It is easy to see that $\sH^{\sT}$ is indeed a double vector subbundle of $\sT\sT P$, therefore it defines a weighted connection on $\sT P$.

Note that the tangent lift $\sH^{(1)}$ of the distribution $\sH$, which is spanned by the vertical and complete lifts of horizontal vector fields on $P$, does not coincide with $H^{\sT}$; we have only $\sH^{(1)}\subset \sH^{\sT}$. The dimension of $\sH^{(1)}$  is four, while the dimension of $\sH^{\sT}$ is five. For the horizontal part $X^h_{\vec{v}}$ of the vector field $X_{\vec{v}}$, we have the vertical lift in the form
$$\left(X^h_{\vec{v}}\right)^{(0)}(p,\vec{u})=(\, p,\,\, u,\, 0,\, \vec{v}-(\vec{n}|\vec{v})\vec{n}\, )\,,$$
and the complete lift
$$\left(X^h_{\vec{v}}\right)^{(c)}(p,\vec{u})=(\, p,\, u,\, \vec{v}-(\vec{n}|\vec{v})\vec{n},\,  \vec{u}\times \vec{v}-(\vec{n}|\vec{u}\times\vec{v})\vec{n}\, ),$$
where, as usual, $n=\pi(p)$.
The distribution $\sH^{(1)}$ reads then
$$\sH^{(1)}=\left\{(p,\vec{u},\vec{v},\vec{w}):\; (\vec{v}|\vec{n})=0=(\vec{w}|\vec{n}),\, \vec{n}=\pi(p) \right\}.$$
It defines a principal connection {\cm on the} bundle $\sT P\rightarrow\sT S^2$ with the action of the tangent group $\sT SO(2)$, however this connection is not weighted.
\end{example}

\small{\vskip1cm}

\noindent Katarzyna GRABOWSKA\\
Faculty of Physics \\
University of Warsaw\\
Pasteura 5,
02-093 Warszawa, Poland
\\Email: konieczn@fuw.edu.pl\\

\noindent Janusz GRABOWSKI\\ Institute of
Mathematics\\  Polish Academy of Sciences\\ \'Sniadeckich 8, 00-656 Warszawa, Poland
\\Email: jagrab@impan.pl \\

\noindent Zohreh RAVANPAK\\ Institute of
Mathematics\\  Polish Academy of Sciences\\ \'Sniadeckich 8, 00-656 Warszawa, Poland
\\Email: zravanpak@impan.pl \\

\section{Declarations}

\textbf{Funding:} Research of JG founded by the  Polish National Science Center grant
under the contract number 2016/22/M/ST1/00542.

\noindent\textbf{Conflicts of interest/Competing interests:} No conflicts of interests.

\noindent\textbf{Availability of data and material:} My manuscript has no associated data.

\noindent\textbf{Code availability:} No code required.

\noindent\textbf{Authors' contributions:} Equal.

{\end{document}